\documentclass{amsart}
\usepackage{amsmath, amssymb}
\newtheorem{Theorem}{Theorem}[section]
\newtheorem{Cor}[Theorem]{Corollary}
\newtheorem{Lemma}[Theorem]{Lemma}
\newtheorem{Proposition}[Theorem]{Proposition}
\theoremstyle{definition}
\newtheorem{Definition}[Theorem]{Definition}
\theoremstyle{remark}
\newtheorem{rem}[Theorem]{Remark}
\newtheorem{ex}[Theorem]{Example}
\numberwithin{equation}{section}
\newcommand{\norm}[1]{\left\Vert\right\Vert}
\newcommand{\abs}[1]{\left\vert\right\vert}
\newcommand{\set}[1]{\left\{\right\}}
\newcommand{\R}{\mathbb R}
\newcommand{\N}{\mathbb N}
\newcommand{\Z}{\mathbb Z}

\newcommand{\C}{\mathbb C}
\newcommand{\F}{\mathcal{F}}

\newcommand{\res}{\hbox{\rm res}}
\newcommand{\tr}{\hbox{\rm tr}}

\begin{document}

\title{On the geometry of $Diff(S^1)-$pseudodifferential operators based on renormalized traces.}%
\author{Jean-Pierre Magnot}%
\address{LAREMA - UMR CNRS 6093 \\ Universit\'e d'Angers \\ 2 Boulevard Lavoisier 
	49045 Angers cedex 01 \\ and \\ Lyc\'ee Jeanne dArc \\ 30 avenue de Grande Bretagne\\
	F-63000 Clermont-Ferrand \\
	 http://orcid.org/0000-0002-3959-3443}%
\email{jean-pierr.magnot@ac-clermont.fr}%

%\date{12/12/2002}

%\commby{}%
% --------------------------------------------------------------

\begin{abstract}
In this article, we examine the geometry of a group of Fourier-integral operators, which is the central extension of $Diff(S^1)$ with a group of classical pseudo-differential operators of any order. Several subgroups are considered, and the corresponding groups with formal pseudodifferential operators are defined. We investigate the relationship of this group with the restricted general linear group $GL_{res},$ we define a right-invariant pseudo-Riemannian metric on it that extends the Hilbert-Schmidt Riemannian metric by the use of renormalized traces of pseudo-differential operators, and we describe classes of remarkable connections.     
\end{abstract}

\maketitle
% ----------------------------------------------------------------


\section*{Introduction}
In the mathematical literature, the introduction of infnite dimensional Lie groups is very often motivated by the use that can be done of such objects which intrinsic properties are difficult to catch. For example, symmetry groups arise in the theory of ordinary differential equations and partial differential equations. Groups of diffeomorphisms  arise in the basic theory of differential manifolds, dynamical systems, and have applications in a wide range of examples such as knot theory, stochastic analysis. Manifolds of maps, current groups and gauge groups have their own applications in various models in physics while their intrinsic properties are deeply related to homotopy theory.
  
The aim of this paper is the description of a family of infinite dimensional Lie groups, derived from a natural interplay between classical pseudodifferential operators and diffeomorphisms, which can be found in the litterature first in \cite{Pay2008} and in \cite{Ma2016}, specializing the study when the base manifold is $S^1.$ Motivations for the introduction of these groups are different in the two works. In \cite{Pay2008} the motivation is the full description of a possible structure group for infinite dimensional principal bundles where a Chern-Weil theory can be stated. In \cite{Ma2016}, the motivation comes from the need of description of the structure group of the space of non-parametrized space of embeddings of a smooth compact boundaryless manifold $M$ to a smooth finite dimensional Riemanian manifold manifold $N.$ In both descriptions, the group under consideration is the central extension of a group of diffeomorphisms by a group of classical pseudo-differential operators. 

More precisely, the group under consideration in a central extension of the group of diffeomoprhisms $Diff(S^1)$ by the group $Cl^*(S^1,V) $ of invertible elements of the algebra $Cl(S^1,V )$ of non-formal, classical, maybe unbounded pseudo-differential operators  acting on a trivial $n-$dimensional Hermitian bundle $S^1 \times V$ over $S^1.$ In such an object, one can derive many infinite dimensional groups: 
\begin{enumerate}
	\item the loop group $C^\infty(S^1,SU_n)$
	\item the group of orientation preserving diffeomorphisms $Diff_+(S^1)$
	\item the group generated by $L^2-$orthogonal symmetries related to $L^2$ orthogonal projections on a finite demensional vector subspace of $C^\infty(S^1,V)$
\end{enumerate}
But beside these examples of groups of bounded operators, which are all subroups of the restricted unitary group $U_{res}$ described in \cite{PS}, we also wish to recover here some non-formal versions of the spaces of elliptic injective classical pseudo differential operators, and in particular the square root of the Laplacian, in order to state results in a framework as general as possible. Due to the presence of unbounded pseudo-differential operators, and in particular differential operators of order 1, the Lie algebra of this group cannot be embedded in a group of bounded operators acting on the space of sections $C^\infty(S^1,V),$ but only represented in it. As a technical remark, we have to say that we have here an example of non-regular infinite dimensional Lie group, as given in remark \ref{nonreg} adapting a remark from \cite{MR2016}. But another technical feature is that this group does not seem to carry atlas, for the same reasons of presence of unbounded operators, as first described in \cite{ARS1} in the context of formal pseudo-differential operators. But one can anyway consider this group as an ``infinite dimensional group'' in the way of \cite{KW}, or as a so-called Fr\"olicher Lie group, see e.g. \cite{MR2016} and references therein, in order to make safe the notion of smoothness. In the section dedicated to the preliminaries, we describe this groups and some of its remarkable subgroups, we recall classical propoerties of the (zeta-))renormalized traces of pseudo-differential operators, describe useful splisstings of the algebra of formal pseudo-differntial operators from existing literature, and fully describe the space of \textbf{formal} $Diff(S^1)-$pseudo-differential operators. This last description is, to our knowledge, not given before this work.

In section \ref{s:res}, we compare this group with the restricted general linear group and develop the index $2-form$ $\lambda$ on it.    
Indeed, by considering only $S^1$ as a base manifold, this enables us to describe more deeply the algebraic and geometric structurees derived from the Dirac operator $D$, or more precisely its sign $\epsilon(D)$, which
 defines a polarization on  $L^2(S^1,\C^k),$ splitting this space into eigenspaces
of positive and non-positive eigenvalues. This polarization is described in e.g. \cite{PS}, is shown here to generalize to unbounded operators in an elegant way, generalizing the remarks initiated in \cite{Ma2003,Ma2006-2} on one hand and some geometric constructions using the Lie group of bounded operators
$$ GL_{res}(S^1,\C^k) = \lbrace u \in GL(L^2(S^1,\C^k)) \hbox{ such that } [\epsilon(D),u] \hbox{ is Hilbert-Schmidt }\rbrace$$
with Lie algebra
$$ \mathcal{L}(S^1,\C^k) = \lbrace u \in \mathcal{L}(L^2(S^1,\C^k)) \hbox{ such that } [\epsilon(D),u] \hbox{ is Hilbert-Schmidt }\rbrace.$$ on the other hand.
This Lie group has a central extension, and the corresponding central extension of its Lie algebra is given by the 
Schwinger cocycle, first found by J. Schwinger in \cite{Sch}. This 2-cocycle was known as a cocycle on Lie algebras of
bounded operators (see e.g. \cite{CDMP}, \cite{Mick}). In \cite{Ma2006-2}, \cite{Ma2008} we proved that the Schwinger
cocycle can be extended naturally to the algebra $PDO(S^1,\C^k)$ of (maybe non classical, maybe unbouded) pseudo-differential operators.
Moreover, we gave its relations with 
a pull-back $c^D_+$ of the Khesin-Kravchenko-Radul cocycle  on formal symbols \cite{KK}, \cite{Rad}, using
an appropriate linear extension of the usual trace of trace-class operators. 
$c^D_+$ and ${1 \over 2}c^D_S$ have the same cohomology class.
On loop groups, this (Lie algebra)-cocycle pulls back to the central extension of the loop algebra \cite{PS},
\cite{CDMP}, \cite{Ma2003}, while it enables to recover the Gelfand-Fuchs cocycle on $Vect(S^1)$ \cite{Ma2006-2}. We extend here the Lie-algebraic considerations of \cite{Ma2006-2} to the grôup under consideration.

In section \ref{s:HS} we construct the extension of the Hlibert-Schmidt Hermitian product
$$(a,b) \mapsto tr(ab^*)$$ to the Lie algebra $Cl(S^1,V) \rtimes Vect(S^1)$ by replacing the usual trace of trace-class operators by one of its linear extensions, the zeta-renormalized trace $\tr^\Delta$ described in the preliminaries. The first difficulty comes from the fact that $\tr^\Delta$ is not tracial, i.e. $$\exists (a,b) \in Cl(S^1;V), \quad \tr^\Delta([a,b])\neq 0.$$
Surprisingly, we obtain a non-degenerate sesquilinear form $(.,.)_\Delta$ on $Cl(S^1;V),$ whicch gives rise to a non-degenerate bilinear form $\mathfrak{Re}(.,.)_\Delta$ on $Cl(S^1;V) \rtimes Vect(S^1).$ We also show that these forms are not positive, and multiplication operators are isotropic vectors. These two forms then generate, by right-invariant action, pseudo-Hermitian and pseudo-Riemannian metrics on $Cl^*(S^1,V)$ and $FCL^*_{Diff(S^1)}(S^1,V)$ respectively.  

Due to the presence of a pseudo-Hermitian metric, one can ask whether there exists pseudo-Hermitian connections. This is the aim of section \ref{s:conn}, where families of pseudo-Hermitian connections are described. However, there exist some technical difficulties to pass in order to describe the whole family of pseudo-Hermitian connections for $(.,.)_\Delta,$ due in particular to the actual lack of knowledge on the linear maps acting on $Cl(S^1,V).$ We restrict our investigations to connection $1-$forms that read as composition by a smoothing operator. Our results are mostly based on the remark that $[a,\epsilon(D)]$ and $sas^*,$ when $s$ is smoothing, are smoothing operators. 
  
  We finish this work by giving some specialization remarks. We describe how $Diff(S^1)$ acts on the polarization, we give another interpretation of the Schwinger cocycle in terms of the curvature of one of our smoothing connections defined in section \ref{s:conn}, we show that a Levi-Civita type exists for $\mathfrak{Re}(.,.)_\Delta$ on the group $FCl^*_{ee,Diff(S^1)}(S^1,V)$ and that it fits with classical formulas for the Levi-Civita connections in finite dimensional settings, and we finish with a blockwise decomposition of  $\mathfrak{Re}(.,.)_\Delta$ on the group of bounded, even-even class operators $FCl^{0,*}_{ee,Diff(S^1)}(S^1,V).$
  
  \vskip 12pt
  \paragraph{Acknowledgements:}
  	The author is happy to dedicate this work to Sylvie Paycha on her 60th birthday, 18 years after working under her supervision, with gratitude for this first step into research.

	{%\color{red}	
{%\color{red}
%	Let us now recall \cite[Proposition
%	1.6.]{Les}, which shows that the distinction between internal tangent cone and internal tangent space is not necessary for diffeological groups.
%	\begin{Proposition} \label{leslie}
%		Let $G$ be a diffeological group. Then the tangent cone at the neutral element $T_eG$ is a diffeological vector space.
%\end{Proposition}}
%Following Iglesias-Zemmour, \cite{Igdiff}, who does not assert that arbitrary diffeological groups have a Lie algebra,
{
	%\color{red}
%	we restrict ourselves to a smaller class of diffeological groups } which have such a tangent space at the neutral element.
{%\color{red} 

\section{Preliminaries on operators on $S^1$}
In this section, we make a moderate use of the notion of diffeology, notion which is not among the main goals of this paper. A classical presentation on diffeologies can be found in \cite{Igdiff}, and the necessary material is already reviewed in many publications of the author. We refer to \cite{MR2019} and references therein for an updated review of the necessary notions to understand in-depth technical properties of the present paper, and to \cite{Ma2016} for a preliminary work about the same objects. When precise notions and properties will be necessary, precise citations will be given. For a superficial reading, the reader can replace diffeologies by a ``natural differentiation" on operators. 
\subsection{Basics on pseudo-differential operators}
An exposition  of basic facts on pseudo-differential operators can
be found in \cite{Gil}. In this section, we only review the definitions and 
tools that are necessary for this note. 

We set $S^1 = \left\{z \in \C \, | \, |z| = 1\right\}$. We shall use for convenience the smooth atlas 
$\mathcal{A}$ of $S^1$ defined as follows:
\begin{eqnarray*}
\mathcal{A} & = & \{\varphi_0,\varphi_1\} ; \\
\varphi_n & : & x \in ]0 ; 2\pi[ \mapsto e^{i(x + n\pi)} \subset S^1 
\hbox{ for } n\in \{0;1\} \end{eqnarray*}     
Associated to this atlas, we fix a smooth partition of the unit $\{s_0;s_1\}$.
We identify each of these functions with its associated multiplication operator when 
necessary.
An operator $A : C^\infty(S^1,\mathbb{C}) \rightarrow C^\infty(S^1,\mathbb{C})$ can be described in 
terms of 4 operators
$$ A_{m,n} : f \mapsto s_m \circ A \circ s_n \hbox{ for } (m,n) \in \{0,1\}^2.$$ 
A scalar
pseudo-differential operator  of order $o$ 
is an operator $$ A : C^\infty(S^1, \mathbb{C}) \rightarrow C^\infty(S^1,\mathbb{C})$$
such that, $\forall (m,n) \in \{0,1\}^2,$
$$ A_{m,n}(f) = \int_{]0 ; 2\pi[} e^{-ix\xi}\sigma_{m,n}(x,\xi) \hat{(s_n.  
f)} (\xi) d\xi$$
where $\sigma_{m,n} \in C^\infty( ]0 ; 2\pi[ \times \R, \mathbb{C})$ satisfies
$$\forall (\alpha, \beta) \in \N^2, \quad |D^\alpha_x D^\beta_\xi \sigma_{m,n}(x,\xi)
|\leq C_{\alpha,\beta}(1 + |\xi|)^{o-\beta}.$$
(In these formulas, the maps $f$, $s_n$ and $A_{m,n}(f)$ are read on the local 
charts $\varphi_0, \varphi_1,$ but we preferred to only mention this aspect and not 
to give heavier formulas and notations, since the setting for $S^1$-pseudo-
differential operators is rather more simple than for manifolds of higher 
dimension.)   

Let $E = S^1 \times \mathbb{C}^k$ be a trivial smooth vector bundle over $S^1$.
Let $s \in \R.$ An operator acting on $C^\infty(S^1,\mathbb{C}^n)$ is a pseudo-differential operator 
if it can be viewed as a $(n \times n)$-matrix of scalar pseudo-differential operators. 
A pseudo-differential operator $A$  of order $o$ 
extends to a linear bounded operator on Sobolev spaces
$H^s(S^1,\C^n) \rightarrow H^{s-o}(S^1,\mathbb{C}^n)$. 
In particular, an order $0$ pseudo-differential operator
is a bounded operator on $H^s(S^1,\mathbb{C}^n)$. 

A pseudo-differential operator of order $o$ is called 
\textbf{classical} if and only if its symbols 
$\sigma_{m,n}$ have an asymptotic expansion
$$ \sigma_{m,n}(x,\xi) \sim_{|\xi| \rightarrow +\infty} \sum_{j=-\infty}^o (\sigma_{m,n})_j(x,\xi),$$
where the maps $(\sigma_{m,n})_j : S^1 \times \R^* \rightarrow \mathbb{C}$, 
called \textbf{partial symbols},
 are j-positively 
homogeneous, i.e. $\forall t>0, (x,\xi) \in S^1 \times \R^*, 
(\sigma_{m,n})_j(x,t\xi) = t^j (\sigma_{m,n})_j(x,\xi).$

We also define \textbf{log-polyhomogeneous} pseudo-differential operators.
A pseudo-differential operator is log-polyhomogeneous if and only if there exists
$o' \in \N$ such that its symbols
$\sigma_{m,n}$ have an asymptotic expansion
$$ \sigma_{m,n}(x,\xi) \sim_{|\xi| \rightarrow +\infty} \sum_{j=0}^o \sum_{k = -\infty}^{o'}
(\sigma_{m,n})_j(x,\xi)(log(|\xi|))^k,$$
where the maps $(\sigma_{m,n})_j : S^1 \times \R^* \rightarrow \mathbb{C}$
are classical partial symbols of order $j$. $o'$ is called the \textbf{logarithmic order} of 
the pseudo-differential operator. Of course, classical pseudo-differential operators are log-polyhomogeneous
pseudo-differential operators of logarithmic order $0$.  

\vskip 5pt
The notion of symbol and partial symbol appear local (dependent on the charts of the 
atlas) in view of these definitions but, in this very special case of atlas on $S^1$ 
where the changes of coordinates are translations,
one can see with the formulas of change of coordinates given in e.g. \cite{Gil} 
that the partial symbols of a pseudo-differential operator can be defined 
globally, taking $$\sigma_j =\sum_{(m,n \in \{0,1\}^2}  (\sigma_{m,n})_j.$$ 
There is another way
to define globally the partial symbols of an operator is in \cite{BK}
, see e.g. \cite{Wid}, 
using linearizations of the manifold. This second way to define the formal symbol 
is more useful when one works with manifolds more complicated than $S^1$. 

Now, define the sets of smooth maps
$$ S(S^1,\mathbb{C}^k) = \bigcup_{o \in \R} S^o(S^1,\mathbb{C}^k)$$
with
$$ S^o(S^1,\mathbb{C}^k) = \left\{ \sigma \in C^\infty(T^*S^1,M_n(\mathbb{C})) \hbox{ such that }\right.$$
$$ \forall (\alpha, \beta) \in \N^2, \quad ||D^\alpha_x D^\beta_\xi \sigma(x,\xi)|| 
\left.\leq C_{\alpha,\beta}(1 + |\xi|)^{o-\beta} \right\}$$
and $$S^{-\infty}(S^1,\mathbb{C}^k) = \bigcap_{o \in \R} S^o(S^1,\mathbb{C}^k).$$
The set $S(S^1,\mathbb{C}^k) / S^{-\infty}(S^1,\mathbb{C}^k)$ can be understood as 
the set of asymptotic expansions when 
$\xi \rightarrow \pm \infty$ up to rapidly decreasing
 maps in the $\xi$ variable. Then, we define the following multiplication
 rule some equivalence classes of maps $\sigma$ and $\sigma'$ in 
$S(S^1,\C^k) / S^{-\infty}(S^1,\C^k)$: 
\begin{equation} \sigma \circ \sigma' = \label{comp2} \sum_{(\alpha)\in \N} {(-i)^\alpha \over \alpha !}
D^\alpha_\xi \sigma D^\alpha_x \sigma'. \end{equation}

\vskip 6pt
\noindent
\textbf{Notations.} 
We note by   $Cl(S^1,\mathbb{C}^k)$) classical pseudo-differential operators acting on smooth
sections of $E$, and by $Cl^o(S^1,\mathbb{C}^k)$ the space of classical 
pseudo-differential operators of order $o$.
\vskip 10pt

There is not an isomorphism between the set of symbols an the set of pseudo-differential operators. 
If we set
$$ Cl^{-\infty}(S^1,\mathbb{C}^k) = \bigcap_{o \in \Z} Cl^o(S^1,\mathbb{C}^k),$$
we notice that it is a two-sided ideal of $Cl(S^1,\mathbb{C}^k)$, and we define the quotient algebra
$$\mathcal{F}Cl(S^1, \C^k) = Cl(S^1,\mathbb{C}^k) / Cl^{-\infty}(S^1,\mathbb{C}^k),$$ 
called the algebra of formal pseudo-differential operators. 
$\mathcal{F}Cl(S^1,\C^k)$ is isomorphic to the set of formal symbols \cite{BK}, 
and the identification is a morphism of $\C$-algebras, for the multiplication on formal symbols
defined before.
\subsection{Renormalized traces} \label{s3}

$E$ is equiped this an Hermitian products $<.,.>$,
which induces the following $L^2$-inner product on sections of $E$: 
$$ \forall u,v \in C^\infty(S^1,E), \quad (u,v)_{L^2} = \int_{S^1} <u(x),v(x)> dx, $$
where $dx$ is the Riemannian volume. 
\begin{Definition}
	$Q$ is a \textbf{weight} of order $s>0$ on $E$ if and only if $Q$ is a classical, elliptic,
	self-adjoint, positive pseudo-differential operator acting on 
	smooth sections of $E$.
\end{Definition}
Recall that, under these assumptions, the weight $Q$ has a real discrete spectrum, and that 
all its eigenspaces are finite dimensional. 
For such a weight $Q$ of order $q$, one can define the complex 
powers of $Q$ \cite{See}, 
see e.g. \cite{CDMP} for a fast overview of technicalities. 
The powers $Q^{-s}$ of the weight $Q$ 
are defined for $Re(s) > 0$ using with a contour integral,
$$ Q^{-s} = \int_\Gamma \lambda^s(Q- \lambda Id)^{-1} d\lambda,$$
where $\Gamma$ is a contour around the real positive axis.
Let $A$ be a log-polyhomogeneous pseudo-differential operator.  The map
$\zeta(A,Q,s) = s\in \C \mapsto \hbox{tr} \left( AQ^{-s} \right)\in \C$ , 
defined for $Re(s)$ large, extends
on $\C$ to a meromorphic function \cite{Le}. When $A$ is classical, 
$\zeta(A,Q,.)$ has a simple pole at $0$  
with residue ${1 \over q} \res_W A$, where $\res_W$ is the Wodzicki
residue (\cite{W}, see also \cite{Ka}). Notice that the Wodzicki residue extends the Adler trace \cite{Adl} on formal symbols.
Following textbooks \cite{PayBook,Scott} for the renormalized trace of 
classical operators, we define

\begin{Definition} \label{d6} Let $A$ be a log-polyhomogeneous pseudo-differential operator. 
	The finite part of $\zeta(A,Q,s)$ at $s = 0$ is called the renormalized trace
	$\tr^Q A$. If $A$ is a classical pseudo-differential operator,
	$$\tr^Q A = lim_{s \rightarrow 0} (\hbox{tr} (AQ^{-s})
	- {1 \over qs} \res_W (A).$$
\end{Definition}

If $A$ is trace class acting on $L^2(S^1,\C^k)$, 
$\hbox{tr}^Q{(A)}=\hbox{tr}{(A)}$.
The functional $\hbox{tr}^Q$ is of course not a trace.
In this formula, it appears that the Wodzicki residue $\res_W(A).$ %which determines the pole coefficient, does not depend on the choice of the weight $Q.$  %The linear map $\res_W$ fulfills the trace property. For more details see \cite{W}.

\begin{Proposition} \label{p5}
	
	\begin{item}
		
		(i) The Wodzicki residue $\res_W$ is a trace on the algebra of
		classical pseudo-differential operators $Cl(S^1,E)$, i.e. $\forall
		A,B \in Cl(S^1,V), \res_W[A,B]=0.$
		
	\end{item}

	\begin{item}
		
		(ii) (local formula for the Wodzicki residue) Moreover, if $A \in Cl(S^1,V)$,
		$$ \res_W A = {1 \over 2\pi} \int_{S^1} \int_{|\xi|=1} tr \sigma_{-1}(x,\xi) d\xi dx = {1 \over 2\pi} \sum_{\xi = \pm 1} \int_{S^1}  tr \sigma_{-1}(x,\xi) d\xi dx. $$
		In particular, $\res_W$ does not depend on the choice
		of $Q$.
	\end{item}

	%\begin{item}
		
	%	(iii) Let $\lambda:Cl(S^1,E) \rightarrow \C$ be trace. If $ dim M
	%	\geq 2$, then $ \exists k \in \C, \lambda=k \res.$
	%\end{item}

\end{Proposition}

Since $\tr^Q$ is a linear extension of the classical trace $\tr$ of trace-class operators acting on $L^2(S^,V),$ it has weaker properties. Let us summarize some of them which are of interest for our work following first \cite{CDMP}, completed by \cite{Ma2016} for the third point.

\begin{Proposition} \label{p6}

	\begin{itemize}
		
		\item   Given two (classical) pseudo-differential operators A and B,
	given a weight Q,
	\begin{equation}\label{crochet}  \tr^Q[A,B] = -{1 \over q} \res (A[B,\log Q]). \end{equation}
	\item Given  a differentiable family $A_t$ of pseudo-differential
	operators, given a differentiable family $Q_t$ of weights of
	constant order q,
	
	\begin{equation}\label{deriv} {d \over dt} \left(tr^{Q_t}A_t\right) = tr^{Q_t} \left({d \over dt}A_t\right) -
	{1 \over q} \res \left( A_t ({d \over dt}\log Q_t) \right).
	\end{equation}

		\item Under the previous notations, if C is a classical elliptic
		injective operator or a diffeomorphism, $tr^{C^{-1}QC}\left( C^{-1}AC
		\right)$ is well-defined and equals $\tr^QA$. 
	
\item Finally, $$\tr^QA = \overline{\tr^{Q^*}A^*}.$$
\end{itemize}
	
\end{Proposition}
 
Since $\tr^Q$ is not tracial, let us give more details on the renormalized trace of the bracket, following \cite{Ma2006}.

\begin{Definition} \label{d8}
	
	Let $E$ be a vector bundle over $S^1,$ let $Q$ a weight and let $a \in \Z$. We define :
	$$ {\mathcal A}^Q_a=\{B \in Cl(S^1,E); [B,\log Q] \in Cl^a(S^1,E)\}.$$
	
\end{Definition}

\begin{Theorem} \label{t1}
	
	\begin{item}
		
		(i) ${\mathcal A}^Q_a \cap Cl^0(S^1,E) $ is an subalgebra of $Cl(S^1,E)$
		with unit.
	\end{item}
	
	\begin{item}
		
		(ii) Let $B \in Ell^*(S^1,E)$, $B^{-1}{\mathcal A}^Q_aB = {\mathcal A}^{B^{-1}QB}_a.$
		
	\end{item}
	
	\begin{item}
		
		(iii) Let $A\in Cl^b(S^1,E)$, and $B \in {\mathcal A}^Q_{-b-2}$,
		then $\tr^Q[A,B]=0.$ As a consequence, $$\forall (A,B) \in Cl^{-\infty}(S^1,V) \times Cl(S^1,V), \quad \tr^Q[A,B]=0.$$
	\end{item}
\end{Theorem}
When needed and appropriate, other properties of renormalized traces will be given later. 

\subsection{Splittings of ${\mathcal F}Cl^{}(S^1,V)$}
\subsubsection{The polarization operator}
The operator $D = {-i} \frac{d}{dx}$ splits $C^\infty(S^1, \C^n)$ into three spaces :

- its kernel $E_0,$ built of constant maps

- $E_+$, the vector space spanned by eigenvectors related to positive eigenvalues

- $E_-$, the vector space spanned by eigenvectors related to negative eigenvalues.

The $L^2-$orthogonal projection on $E_0$ is a smoothing operator, which has null formal symbol. By the way, concentrating our attention on the formal symbol of operators first, we can ignore this projection and hence we work on $E_+ \oplus E_-$. The following elementary result will be useful for the sequel.  
\begin{Lemma} \label{l1} \cite{Ma2003,Ma2006-2}
	
	(i) $\sigma(D) = {\xi }$
	
	(ii) $\sigma(|D|) = {|\xi| }$
	
	(iii)  $\sigma(\epsilon) = {\xi \over |\xi|}$, where $\epsilon = D|D|^{-1} = |D|^{-1}D$ is the sign of D.
	
	(iv)  Let $p_{E_+}$ (resp. $p_{E_-}$) be the projection on $E_+$ (resp. $E_-$), then 
	$\sigma(p_{E_+}) ={1 \over 2}(Id + {\xi \over |\xi|})$ and $\sigma(p_{E_-}) = {1 \over 2}(Id - {\xi \over |\xi|})$.
\end{Lemma}

\begin{rem}
	The operators $D$ and $\epsilon$ are operators acting on $L^2(S^1,\C)$ while $\frac{d}{dx}$ and $i\epsilon$ are operators on $L^2(S^1,\C)$ which leave invariant $L^2(S^1,\R).$
\end{rem}
%\begin{proof}

Let us now give a trivial but very useful lemma:
\begin{Lemma}\label{l2} \cite{Ma2003}
	Let $f: \R^* \rightarrow V$ be a 0-positively homogeneous function with values in a topological vector space $V$. Then, for any $n \in \N^*$, $f^{(n)} = 0$ where $f^{(n)}$ denotes the n-th derivative of $f$.
\end{Lemma}

%\begin{proof}
%Let $x \in \R^*$, and $t >0$. Since $f$ is 0-positively homogeneous, $f(tx) = f(x)$. Hence, $f$ is constant on $\R^+$ and on $\R^-$. This yields the result. \end{proof}

\begin{Cor} \label{0-hom}
	Let $A \in \F Cl(S^1,\C^n)$ such that $\sigma(A) = \sigma_0(A).$ then $$\sigma(A\circ B) = \sigma_0(A). \sigma(B) \hbox{ (pointwise multiplication).}$$
\end{Cor}

\subsubsection{The splitting with induced by the connected components of $T^*S^1-S^1.$} 
In this section, we define two ideals of the algebra $\mathcal{F}Cl(S^1,V)$, 
that we call $\mathcal{F}Cl_+(S^1,V)$ and $\mathcal{F}Cl_-(S^1,V)$, such that $\mathcal{F}Cl(S^1,V) = \mathcal{F}Cl_+(S^1,V) \oplus \mathcal{F}Cl_-(S^1,V)$. 
This decomposition is explicit in \cite[section 4.4., p. 216]{Ka}, and we give an explicit description here following \cite{Ma2003,Ma2006-2}. 

\begin{Definition}
	
	Let $\sigma$ be a partial symbol of order $o$ on $E$. Then, we define, for $(x,\xi) \in T^*S^1 - S^1$, 
	$$ \sigma_+(x,\xi) = \left\{ 
	\begin{array}{ll}
	\sigma(x,\xi) & \hbox{ if $ \xi > 0$} \\
	0 & \hbox{ if $ \xi < 0$} \\
	\end{array}
	\right. \hbox{ and }
	\sigma_-(x,\xi) = \left\{ 
	\begin{array}{ll}
	0 & \hbox{ if $ \xi > 0$} \\
	\sigma(x,\xi) & \hbox{ if $ \xi < 0$} . \\
	\end{array}
	\right.$$
	We define  $p_+(\sigma) = \sigma_+$ and $p_-(\sigma) = \sigma_-$ .
\end{Definition}
The maps 
$ p_+ : \mathcal{F}Cl(S^1,V) \rightarrow \mathcal{F}Cl(S^1,V) $ { and } $p_- : \mathcal{F}Cl(S^1,V) \rightarrow \mathcal{F}Cl(S^1,V)$ are clearly smooth algebra morphisms (yet non-unital morphisms) that leave the order invariant and are also projections (since multiplication on formal symbols is expressed in terms of pointwise multiplication of tensors). 

\begin{Definition} We define
	$  \mathcal{F}Cl_+(S^1,V) = Im(p_+) = Ker(p_-)$
	and $  \mathcal{F}Cl_-(S^1,V) = Im(p_-) = Ker(p_+).$ \end{Definition}
Since $p_+$ is a projection,  we have the splitting
$$ \mathcal{F}Cl(S^1,V) = \mathcal{F}Cl_+(S^1,V) \oplus \mathcal{F}Cl_-(S^1,V) .$$
Let us give another characterization of $p_+$ and $p_-$. 
Looking  more precisely at the formal symbols of $p_{E_+}$ and $p_{E_-}$ computed in Lemma \ref{l1}, we observe that 
$$ \sigma( p_{E_+}) = \left\{ \begin{array}{ll}
1 & \hbox{if }\xi > 0 \\
0 & \hbox{if }\xi < 0 \\
\end{array} \right. \hbox{ and }
\sigma( p_{E_-}) = \left\{ \begin{array}{ll}
0 & \hbox{if }\xi > 0 \\
1 & \hbox{if }\xi < 0 \\
\end{array} \right. . $$
In particular, we have that $p_+$ and $p_-$ satisfy Corollary \ref{0-hom}. Moreover, their symbol do not depend on $x.$
From this, we have the following result.

\begin{Proposition} \label{pag0}
	Let $A \in \mathcal{F}Cl(S^1,V).$ 
	$ p_+(A) =  \sigma( p_{E_+}) \circ A = A \circ \sigma( p_{E_+})$ and 
	$  p_-(A) =  \sigma( p_{E_-}) \circ A = A \circ \sigma( p_{E_-}).$
\end{Proposition}

%\begin{proof}
%	The proof follows from the fact that $\sigma( p_{E_+}) \circ A = A. 1_{\xi > 0}$. \end{proof}

\subsubsection{The ``odd-even'' splitting}

We note by $\sigma(A)(x,\xi)$ the total formal symbol of $A \in \mathcal{F}Cl(S^1,V).$ 
The following proposition is trivial:
\begin{Proposition}
	Let 
	$\phi:\mathcal{F} Cl(S^1,V)\rightarrow \mathcal{F}Cl(S^1,V)$ defined by 
	$$\phi(A) = \frac{1}{2}\sum_{k \in \mathbb{Z}}\sigma_{k}(A)(x,\xi) - (-1)^k\sigma_{k}(A)(x,-\xi).$$
	This map is smooth, and $$\Psi DO(S^1,V)=\mathcal{F} Cl_{ee}(S^1,V) = Ker(\phi).$$
\end{Proposition}

Following \cite{Scott}, one can define \textbf{even-odd class} pseudo-differential operators 
$$\mathcal{F}Cl_{eo}(S^1,V) = \left\{ A \in\mathcal{F} Cl(S^1,V) \, | \, \sum_{k \in \mathbb{Z}}\sigma_{k}(A)(x,\xi) + (-1)^{k}\sigma_{k}(A)(x,-\xi) = 0 \right\}.$$
\begin{rem}
	In e.g. \cite{KV1,KV2}, even-even class pseudodifferential operators are called odd class pseudodifferential operators. By the way, following the terminology of \cite{KV1,KV2} even-odd class pseudo-differential operators should be called even class. In this paper we prefer to fit with the terminology given in the textbooks\cite{PayBook,Scott} %even if the initial terminology given in \cite{KV1,KV2} and its natural extension would appear more natural to us.
\end{rem}

\begin{Proposition}
	$\phi$ is a projection and $\mathcal{F}Cl_{eo}(S^1,V) = Im \phi.$
\end{Proposition}
By the way, we also have
$$ \mathcal{F}Cl(S^1,V) = \mathcal{F}Cl_{ee}(S^1,V) \oplus \mathcal{F}Cl_{eo}(S^1,V).$$

We have the following composition rules for the class of a formal operator $A\circ B:$
\vskip 12pt
\begin{tabular}{|c|c|c|}
	\hline
	&&\\
	& $A$ even-even class & $A$ even class \\
	&&\\
	\hline
	&&\\
	$B$ even-even class & $A \circ B$ even-even class & $A \circ B$ even-odd class \\
	&&\\
	\hline
	&&\\
	$B$ even-odd class & $A \circ B$ even-odd class & $A \circ B$ even-even class \\
	&&\\
	\hline
\end{tabular}

\begin{ex}
	$\epsilon(D)$ and $|D|$ are even-odd class, while we already mentionned that differential operators are even-even class.
\end{ex}

Applying the local
formula for Wodzicki residue, one can prove \cite{CDMP}:

\begin{Proposition} \label{p8}
	
	If $A$ and $Q$ lie in the odd
	class, then $f(s)=tr(AQ^{-s})$ has no pole at $s=0$. Moreover, if
	A and B are odd class pseudo-differential operators, $\tr^Q \left(
	[A,B] \right) =0$ and $\tr^QA$ does not depend on Q.
	
\end{Proposition}
\subsection{Formal and non-formal $Diff(S^1)-$pseudodifferential operators}
It follows from \cite{Ee,Om} that $Diff_+(S^1)$ is open in the 
Fr\'echet manifold $C^\infty(S^1,S^1)$. This fact makes it a Fr\'echet
manifold and, following \cite{Om}, a regular Fr\'echet Lie group.
In addition to groups of pseudo-differential operators, we also need {a restricted class of} groups of 
Fourier integral operators which we will call $Diff(S^1)-$pseudodifferential operators following \cite{Ma2016,MR2018}.
These groups appear as central extensions of $Diff(S^1)$ or $Diff_+(S^1)$ by groups of pseudo-differential operators. 
We do not state the basic facts on Fourier integral operators here (they can be found in the classical paper \cite{Horm}). 
The pseudo-differential operators considered here can be classical, odd class, or anything else. Applying 
the formulas of ``changes of coordinates'' (which can be understood as adjoint 
actions of diffeomorphisms) of e.g. \cite{Gil}, we obtain that even-even and even-odd class pseudo-differential operators are stable 
under the adjoint action of $Diff(S^1).$  Thus, we can define the following groups \cite{Ma2016}:

\begin{Definition}
	\begin{enumerate}
		\item The group $FCl_{Diff(S^1)}^{*}(S^1,V)$ is the infinite dimensional group defined by 
		$$FCl_{Diff(S^1)}^{*}(S^1,V) = \left\{ A = B \circ g \, | \, B \in Cl^*(S^1,V) \hbox{ and } g \in Diff(S^1) \right\}.$$
		\item The group $FCl_{Diff(S^1)}^{0,*}(S^1,V)$ is the infinite dimensional group defined by 
		$$FCl_{Diff(S^1)}^{0,*}(S^1,V) = \left\{ A = B \circ g \, | \, B \in Cl^{0,*}(S^1,V) \hbox{ and } g \in Diff(S^1) \right\}.$$
		\item The group $FCl_{Diff(S^1),ee}^{*}(S^1,V)$ is the infinite dimensional group defined by 
		$$FCl_{Diff(S^1),ee}^{*}(S^1,V) = \left\{ A = B \circ g \, | \, B \in Cl^*_{ee}(S^1,V) \hbox{ and } g \in Diff(S^1) \right\}.$$
		\item The group $FCl_{Diff(S^1),ee}^{0,*}(S^1,V)$ is the infinite dimensional group defined by 
		$$FCl_{Diff(S^1),ee}^{0,*}(S^1,V) = \left\{ A = B \circ g \, | \, B \in Cl^{0,*}_{ee}(S^1,V) \hbox{ and } g \in Diff(S^1) \right\}.$$
	\end{enumerate}
\end{Definition}

%The same definitions hold whe we replace the group $Diff(S^1)$ by its connected component of the unit $Diff_+(S^1)$ made of orientation perserving diffeomorphisms. We the get the subgroups of $FCl_{Diff(S^1)}^{*}(S^1,V)$ that we note by $FCl_{Diff_+(S^1)}^{*}(S^1,V),$ $FCl_{Diff_+(S^1)}^{0,*}(S^1,V),$ $FCl_{Diff_+(S^1),ee}^{*}(S^1,V)$ and $FCl_{Diff_+(S^1),ee}^{0,*}(S^1,V)$ with obvious notations.
 
\begin{rem} \label{Baaj} 
	This construction of phase functions of $Diff(S^1)-$pseudo-differential operators
	differs from the one described by Omori \cite{Om} 
	and Adams, Ratiu and Schmid \cite{ARS2} for some groups of Fourier integral operators; the exact relation among these
	constructions still needs to be investigated.
\end{rem}

\begin{rem}
	 The decomposition $A = B \circ g$ is unique \cite{Ma2016}, and the diffeomorphism appears as the phase of the Fourier integral operator. 	
\end{rem}

\begin{rem} the group $Diff(S^1)$ decomposes into two connected components 
	$Diff(S^1) = Diff_+(S^1) \cup Diff_-(S^1)\; ,$
	where the connected component of the identity, $Diff_+(S^1)$, is the group of orientation preserving diffeomorphisms 
	of $S^1$. By the way, we can replace $Diff(S^1)$ by $Diff_+(S^1)$ in the previous definition.
\end{rem}

\begin{rem} We have, on these Lie groups, somme difficulties to exhibit an atlas especially when considering unbounded operators. One then can consider "natural" notions of smoothness, inherited from the embedding into $Cl(S^1,V)$ for the pseudo-differential part, and from the well-known structure of ILB  Lie group \cite{Om} from the diffeomorphism (phase) component. In order to be more rigorous, one can then consider Fr\"olicher Lie groups along the lines of \cite{Ma2016,MR2018} in this context, or in \cite{Ma2018-2,MR2016} when dealing with other examples where this setting is useful. A not-so-complete description of technical properties of Fr\"olicher Lie groups can be found in works by other authors \cite{BN2005,Lau2011,Neeb2007} but this area of knowledge, however, still needs to be further developed. This lack of theoretical knowledge is not a problem for understanding and dealing with groups of $Diff(S^1)-$pseudo-differential operators.
\end{rem}
From this last remark, we deduce that the Lie algebra of $FCl_{Diff(S^1)}^{*}(S^1,V)$ is $Cl(S^1,V) \oplus Vect(S^1),$ and we remark that $$(a,X) \mapsto a + X$$ is a Lie algebra morphism with values in $Cl(S^1,V).$
The same remark holds for subgroups of $FCl_{Diff(S^1)}^{*}(S^1,V).$
Let us now define a relation of equivalence "up to smoothing operators". 

\begin{Definition}
	Let $(A,A') \in (FCl_{Diff(S^1)}^{*}(S^1,V))^2,$ with $A = B \circ g$ and $A' = B' \circ g'$ as before. Then $$A \equiv A' \quad \Leftrightarrow \quad \left\{\begin{array}{rcl} g = g' && \\
	B - B' & \in & Cl^{-\infty}(S^1,V) \end{array} \right.$$
	The set of equivalence classes with respect to $\equiv$ is noted as $\mathcal{F}FCl_{Diff(S^1)}^{*}(S^1,V)$ and is called the set of formal $Diff(S^1)-$pseudodifferential operators.
\end{Definition}

The same spaces of formal operators can be constructed using orientation-preserving diffeomorphisms of $S^1,$ even-even class pseudodifferential operators and so on. We do not feel the need to give here redundant constructions, and obvious notations.

\begin{Theorem} 
	Let $$G = \left\{ A \in Cl^{0;*}(S^1,V) \, | \, A = Id + B, \, B \in Cl^{-\infty}(S^1,V)\right\}.$$
	Then \begin{itemize}
		\item $G \vartriangleleft FCl^*_{Diff(S^1)}(S^1,V),$
		\item given $(A,A') \in FCl^*_{Diff(S^1)}(S^1,V)^2,$
		$$ A \equiv A' \Leftrightarrow AA'^{-1} \in G$$
	\end{itemize}
	 which implies that $$\mathcal{F}FCl_{Diff(S^1)}^{*}(S^1,V)= FCl_{Diff(S^1)}^{*}(S^1,V)/G.$$ By the way, $\mathcal{F}FCl_{Diff(S^1)}^{*}(S^1,V)$ is a group. Moreover, \begin{equation} \label{str} \mathcal{F}FCl_{Diff(S^1)}^{*}(S^1,V) = \mathcal{F}Cl^{*}(S^1,V) \rtimes Diff(S^1).\end{equation}
\end{Theorem} 

\begin{proof} Let $(A,A') \in FCl^*_{Diff(S^1)}(S^1,V),$ with $A = B \circ g, $ $A' = B' \circ g'$ in the decomposition $ FCl^*_{Diff(S^1)}(S^1,V) =  Cl^*_{}(S^1,V) \rtimes Diff(S^1).$ 
	\begin{eqnarray*}
		&&A \equiv A'\\ 
		& \Leftrightarrow &  \left\{\begin{array}{lcl} g = g' &&\\
			B - B'=-R & \in & Cl^{-\infty}(S^1,V) \end{array} \right. \\
		& \Leftrightarrow & A A'^{-1} = B B'^{-1} = Id + RB'^{-1} \in G.
	\end{eqnarray*}

	Let $A \in FCl^*_{Diff(S^1)}(S^1,V)$ and let $Id+S \in G$ with $S \in Cl^{-\infty}(S^1,\C).$ Then $$A^{-1}(Id + S)A = Id + A^{-1}SA.$$
	We have $$A = B \circ g $$ with $B \in Cl^*(S^1,V)$ and $g \in Diff(S^1).$ 
	As an operator acting on $L^2(S^1,V)$-functions, composition on the right by a diffeomorphism is a bounded operator. As an operator from $C^\infty(S^1,V)$ to $C^\infty(S^1,V),$ composition on the right by a diffeomorphism is also a bounded operator. A pseudo-differential operator of order $o$ is a bounded operator from $L^2(S^1,V)$ to $H^{-o}(S^1,V),$ and also from $C^\infty(S^1,V)$ to $C^\infty(S^1,V).$ Finally a smoothing operator is bounded from any Sobolev space $H^{-o}$ to $C^\infty.$ By the way, $$A^{-1} S A = g^{-1} B^{-1} S B g \in Cl^{-\infty}(S^1,V).$$ Hence $G \vartriangleleft FCl^*_{Diff(S^1)}(S^1,V).$ Let us now examine the following diagram: 
	
	\begin{equation} \label{diagram1} \begin{array}{ccccccccc}
		&&  && 1 && 1 &&\\
		&&  && \downarrow && \downarrow &&\\
		1 & \rightarrow & G & \rightarrow & Cl^*(S^1,V) & \rightarrow & \F Cl^*(S^1,V) & \rightarrow & 1 \\
		&& \| && \downarrow && \downarrow &&\\ 
		1 & \rightarrow & G & \rightarrow & FCl^*_{Diff(S^1)}(S^1,V) & \rightarrow & \F FCl^*_{Diff(S^1)}(S^1,V) & \rightarrow & 1 \\
		&&  && \downarrow && \downarrow &&\\
			&&  && Diff(S^1) & = & Diff(S^1) &&\\
				&&  && \downarrow && \downarrow &&\\
			&&  && 1 && 1 &&\\
		\end{array} \end{equation}
		The three squares commute, the two horizontal lines are short exact sequences as well as the central culumn. By the way, we also have that 
		$$ 1 \rightarrow  \F Cl^*(S^1,V) \rightarrow \F FCl^*_{Diff(S^1)}(S^1,V) \rightarrow Diff(S^1) \rightarrow 1$$ is a short exact sequence. 
\end{proof}

Via identification \ref{str}, $\mathcal{F}FCl_{Diff(S^1)}^{*}(S^1,V)$ as a (set theoric) product can be equiped by the product topology of $\mathcal{F}Cl^{*}(S^1,V) \times Diff(S^1).$ This makes of it a smooth Lie group modelles on a locally convex topological vector space, and we can state:

\begin{Proposition}
	There is a natural structure of infinite dimensional Lie group on $\mathcal{F}FCl_{Diff(S^1)}^{*}(S^1,V)$, and
its Lie algebra (defined by germs of smooth paths) reads as $$\mathcal{F}Cl(S^1,V) \rtimes Vect(S^1).$$ %Its Lie bracket  is $$[(u,X),(v,Y)] = \left( [u,Y] + [X,v] + [u,v], [X,Y]\right).$$
\end{Proposition}

\begin{rem}
	\label{nonreg}
	It is proven in \cite{Ma2016} that, in an algebra of formal pseudo-differential operators that can be identified with $\F Cl_{ee}(S^1,\C)$ in our context, the constant vector field $t \mapsto \frac{d}{dx}$ does not integrate to a smooth path on the group, in other words $$exp\left(\frac{d}{dx}\right)\notin \F Cl_{ee}^*(S^1,\C).$$ This shows that $\F Cl_{ee}^*(S^1,\C),$ and hence $Cl^*(S^1,V)$ and also $FCl^*_{Diff(S^1)}(S^1,V)$ are not regular in the sense of Omori \cite{Om}, while the same constant vector, understood as an element of $Vect(S^1),$ the Lie algebra of $Diff(S^1),$ integrates to a rotation on the circle. This shows that one has to be careful on which component the differential monomials of degree $1$ are considered while working with $FCl^*_{Diff(S^1)}(S^1,V).$ Its Lie algebra cannot be embedded but only represented in $Cl_{ee}(S^1,V).$
\end{rem}

\section{Relation with the restricted linear group} \label{s:res}

\subsection{On cocycles on $Cl(S^1,\mathbb{C}^k)$} \label{sect2}

Let us first precise which polarization we choose on $L^2(S^1,\mathbb{C}^k)$. 
We can choose independently two polarizarions : 

- one setting $H_+^{(1)} = E_+$ and $H_-^{(1)} = E_0 \oplus E_-$,
 
- or another one setting $H_+^{(2)} = E_+ \oplus E_0$ and $H_-^{(2)} =  E_-$. 

Since $E_0$ is of dimension $k$, the orthogonal projection on $E_0$ is a 
smoothing operator. Hence, 
$$\sigma(p_{H_+^{(1)}}) =  \sigma(p_{H_+^{(2)}}) = 1_{\xi > 0}$$
and 
$$\sigma(p_{H_-^{(1)}}) =  \sigma(p_{H_-^{(2)}}) = 1_{\xi < 0}.$$
 
We introduce the notation, for $A \in PDO(S^1, \mathbb{C}^k)$, 
 $$A_{++} = p_{H_+} A p_{H_+},$$
where $H_+$ denotes $H_+^{(1)}$ or $H_+^{(2)}$, and we set 
$\epsilon(D) = p_{H_+} - p_{H_-}$.
We notice that $\sigma(A_{++}) = \sigma_+(A)$, and recall the following result \cite{Ma2006-2}:  

\begin{Theorem} \label{th1}
For any $A \in Cl(S^1,\mathbb{C}^k)$, $[A,\epsilon(D)] \in Cl^{-\infty}(S^1,\mathbb{C}^k).$ Consequently, 
$$ c_s^D : A,B \in Cl(S^1,E) \mapsto {1 \over 2}\tr \left( \epsilon(D)[\epsilon(D),A][\epsilon(D),B]  \right)  = \tr\left(\left[A,\epsilon(D)\right]B \right)$$
is a well-defined  2-cocycle on $PDO(S^1, \mathbb{C}^k).$ Moreover, $c_s^D$ is non trivial on any Lie algebra $\mathcal A$ such that $C^\infty(S^1,\mathbb{C}^k)\subset \mathcal{A} \subset Cl(S^1,\mathbb{C}^k).$
\end{Theorem}
Along this cocycle, we have to mention two others. First, the Kravchenko-Khesin cocycle \cite{KK}, defined on Adler series $a = \sum_{n \leq k} a_n \xi^n$ and $b= \sum_{m\leq l} b_m \xi^m$ by
$$c_{KK}(a,b) = \res \left(a[b,\log\xi]\right)$$
which pulls-back on the algebra $Cl(S^1,\C^k),$ following a procedure first described by Radul \cite{Rad} to our knowledge, to a cocycle called Kravchenko-Khesin-Radul cocycle in \cite{Ma2006-3}, defined in its polarized version by
$$ c_{KKR}(A,B) = \tr^{|D|}\left([A_{++},B_{++}]\right) = \frac{1}{2\pi}\res_W\left(A_{++}[B_{++}, log|D|]\right),$$
 which is the pull-back, up to the constant $\frac{1}{2\pi},$ of $c_{KK}$ through the map $$A \in Cl(S^1,\C^k)\mapsto \sigma_+(A).$$
 Secondly, the index cocycle described in \cite{PS} and extended to $Cl(S^1,\C^k)$ in \cite{Ma2006-2}, defined  by:
 $$\lambda(A,B) = \tr\left([A_{++},B_{++}] - [A,B]_{++}\right).$$
 We have to remark that, in order to have a well-defined formula, the operator
 $[A_{++},B_{++}] - [A,B]_{++}$ needs to be trace-class. When $A$ and $B$ are pseudo-differential operators, this operator is smoothing and hence the trace is well-defined.
 Following \cite{Ma2006-2}, we can state:
 \begin{Proposition}
 	On $PDO(S^1,\mathbb{C}^k),$ $$\lambda = \frac{1}{2}c_s.$$ Moreover, on any algebra $\mathcal{A}$ such that $C^{\infty}(S^1,\mathbb{C}^k)\subset \mathcal{A} \subset Cl(S^1,\mathbb{C}^k),$
 	the cocycles $c_{KKR}, \lambda$ and $\frac{1}{2}c_s$ are non trivial in Hoschild cohomology, and belong to the same cohomology class. 
 \end{Proposition}
\subsection{$GL_{res}$ and its subgroups of Fourier-integral operators}
Let us now turn to the Lie group of bounded operators described in \cite{PS}:
$$ GL_{res}(S^1,\mathbb{C}^k) = \lbrace u \in GL(L^2(S^1,\mathbb{C}^k)) \hbox{ such that } [\epsilon(D),u] \hbox{ is Hilbert-Schmidt }\rbrace$$
with Lie algebra
$$ \mathcal{L}(S^1,\mathbb{C}^k) = \lbrace u \in \mathcal{L}(L^2(S^1,\mathbb{C}^k)) \hbox{ such that } [\epsilon(D),u] \hbox{ is Hilbert-Schmidt }\rbrace.$$
 
 \begin{Proposition} \cite[Theorem 25-26]{Ma2016}
 	$FCl_{Diff_+(S^1)}^{0,*}(S^1,\C^k) \subset GL_{res}(S^1,\mathbb{C}^k)$
 	\end{Proposition}
 	
 Let us now give a new light on an old result present in \cite{PS} from a topological viewpoint, expressed by remarks but not stated clearly in the mathematical litterature to our knowledge. We choose here a new approach for the proof, more easy and much more fast, and adapted to our approach of (maybe generalized) differentiability pior to topological considerations. 
 \begin{Lemma}
 	The injection map $ Diff_+(S^1) \hookrightarrow GL_{res}(S^1,\mathbb{C}^k)$
 	is not differentiable.
 \end{Lemma}
\begin{proof} Let us assume that the injection map $ Diff_+(S^1) \hookrightarrow GL_{res}(S^1,\mathbb{C}^k)$ is differentiable. The group $GL_{res}(S^1,\mathbb{C}^k)$ is acting smoothly on $L^2(S^1,\C^n)$ and hence the Lie algebra of $GL_{res}(S^1,\mathbb{C}^k)$ is a Lie algebra of bounded operators acting on $L^2(S^1,\C^n).$ The Lie algebra $Vect(S^1)$ is a Lie algebra of unbounded operators acting on $L^2(S^1,\C^n)$ hence the injection map $ Diff_+(S^1) \hookrightarrow GL_{res}(S^1,\mathbb{C}^k)$ is not differentiable. 
	\end{proof}
From this Lemma, next theorem is straightforward:
 \begin{Theorem}
 The injection maps 		$FCl_{Diff_+(S^1)}^{0,*}(S^1,\C^k) \hookrightarrow GL_{res}(S^1,\mathbb{C}^k)$
 and $DO^{0,*}(S^1,\mathbb{C}^k) \rtimes Diff_+(S^1) \hookrightarrow GL_{res}(S^1,\mathbb{C}^k)$ are not differentiable.
 	\end{Theorem}

 \subsection{Index cocycle on $FCl_{Diff(S^1)}^{0,*}(S^1,\C^k)$}
 Let us now recall a result from \cite{PS}:
 \begin{Proposition}
 	Let us consider the right-invariant 2 form generated by $\lambda$ on $GL_{res},$ that we note by $\tilde \lambda.$ Then $\tilde \lambda$ is a closed, non exact 2-form on $GL_{res}.$
 \end{Proposition}
 \begin{Lemma}
 	Let $\gamma: \R \rightarrow FCl_{Diff_+(S^1)}^{*}(S^1,\C^k)$ be a smooth path. Then the path $$ t \in \R \mapsto \left[\epsilon(D),\gamma(t)\right]$$ is a smooth path of smoothing operators with respect to any of these differentiable structures: 
 	\begin{itemize}
 		\item the differentiable structure of $FCl_{Diff_+(S^1)}^{*}(S^1,\C^k)$
 		\item  the differentiable structure of $GL_{res}(S^1,\mathbb{C}^k).$
 	\end{itemize}
 \end{Lemma}
 \begin{proof}
 	The proof follows from $\left[FCl_{Diff_+(S^1)}^{*}(S^1,\C^k),\epsilon(D)\right]\subset Cl^{-\infty}(S^1,\C^k)$ as is stated in \cite[Theorem 25]{Ma2016}.
 \end{proof}
As a consequence, we get: 

\begin{Theorem}
	$\lambda = \frac{1}{2}c_s$ generates a closed, non exact 2-form on $FCl_{Diff_+(S^1)}^{*}(S^1,\C^k)$.
\end{Theorem}

 \section{Renormalized extension of the Hilbert-Schmidt Hermitian metric}\label{s:HS}
 The vector space $Cl^{-1}(S^1,V)$ is a space of Hilbert-Schmidt operators. As a subspace,  $Cl^{-1}(S^1,V)$ inherits a Hermitian metric from the classical Hilbert-Schmidt inner product. The renormalized trace $\tr^\Delta$ extends the classical trace $\tr$ of trace class operators to a smooth linear functional on $Cl(S^1,V).$ We investigate here the possible (maybe naive) extension of the classical Hilbert-Schmidt inner product to $Cl(S^1,V)$ via $\tr^\Delta.$
 \subsection{Calculation of renormalized traces}
Let  $(z^k)_{k \in \mathbb{Z}}$ is the Fourier $L^2-$orthonormal basis. Let us recall that there exists an ambiguity on $\epsilon(D)$ concerning its action on $z^0,$ which can be, or not, in the kernel of $p_+,$ or in the eigenspace of the eigenvalue $1$ or $-1.$ Depending on each of these three possibilities respectively, we set $\epsilon(k)$ as the eigenvalue of $\epsilon(D)$ at the eigenvector $z^k.$
 \begin{Lemma} \label{calcultr}
 	Let $X=u\frac{d}{dx},Y=v\frac{d}{dx}$ be two vector fields over $S^1,$ and let $a,b \in >C^\infty(S^1,\mathbb{C}).$ 
 	Then
 	\begin{enumerate}
 	%	\item $$\tr^{\Delta}(\epsilon(D)a\bar{b}) =  = \epsilon(0) (a,b)_{L^2}$$
 	%	\item $$\tr^{\Delta}(a\epsilon(D)\bar{b}) = \left(a, (\epsilon(0) Id + 2D)b\right)_{L^2}$$
 	%	\item $$\tr^{\Delta}(X\epsilon(D) {Y}^*)=- \left(\left(\frac{2}{3}D^3 + \frac{1}{3}D \right)u,v\right)_{L^2}$$
 	%	\item $$ \tr^{\Delta}(X\epsilon(D) \bar{a})=\frac{i}{12} (u,a)_{L^2} + 2i (u,\Delta a)_{L^2}$$ and $$\tr^\Delta(a \epsilon(D)X^*)=-\frac{i}{12} (u,a)_{L^2} - 2i (u,\Delta a)_{L^2}$$
 	%	\item $$\tr^{\Delta}(p_+ a\bar{b}) = \left( \epsilon(0)  - \frac{1}{2} \right) (a,b)_{L^2}$$
 		\item \label{3.} $$\tr^{\Delta}(a\bar{b}) = 0$$
 		\item \label{4.} $$ \tr^{\Delta}(XY^*) = 0 $$ %-2\zeta(-2)(u,v)_{L^2}$$
 		\item $$ \tr^{\Delta}(Xa)=\tr^\Delta(aX)=0$$ 
 	%	\item $$\tr^{\Delta}(p_+ XY)=\tr^\Delta\left(p_+D\right)\left(\frac{du}{dx},v\right)_{L^2}   -\zeta(2) (u,v)_{L^2}$$ and $$\tr^{\Delta}(\epsilon(D) XY)=	2\tr^\Delta\left(p_+D\right)\left(\frac{du}{dx},v\right)_{L^2}.$$	
 	\end{enumerate} 
 	\end{Lemma}
 \begin{proof}
 	We have here even-even class operators, so that the renormalized trace is commuting in all items. 
 	\begin{enumerate}
 		\item Following computations present e.g. in \cite{CDMP},
 			we compute for $Re(s)$ large enough first, with the convention $0^{-s}=1:$ %and where $\epsilon(k)$ is the eigenvalue of $\epsilon(D)$ at the eigenvector $z^k:$
 		 \begin{eqnarray*}
 			\sum_{k \in \mathbb{Z}}((z^n z^{-m} z^k),\Delta^{-s} z^k)_{L^2}	& = &\sum_{k \in \mathbb{Z}} (k^2)^{-s} (z^n z^{-m} z^k, z^k)_{L^2} \\ & = &  \sum_{k \in \mathbb{Z}} (k^2)^{-s} (z^{n-m+k}, z^k)_{L^2} \\
 				& = & \left\{ \begin{array}{cl} 0 & \hbox{ if } n \neq m \\
 				 1 + 2\sum_{k \in \mathbb{N}^*} (k^2)^{-s} & \hbox{ if } n=m \end{array} \right.
 				\end{eqnarray*}
 			By the way, taking the limit as $s \rightarrow 0$ in the zeta-renormalization procedure of the trace, we get, dividing the sum for $k \in \Z$ into three parts: $k \in \Z_-^*,$ $k=0$ and $k \in \N^*$: $$\tr^{\Delta}(a\bar{b}) = \left( \zeta(0) + 1 + \zeta(0) \right) (a,b)_{L^2} = 0 (a,b)_{L^2}.$$
 		\item Here the functions $u$ and $v$ are real-valued, which means that one should consider the real Fourier basis for the summation. However, since the real Fourier basis is a linear combination of the complex one, we investigate first the renormalized trace with $u = z^n$ and $v = z^m.$ Then $XY^*(z^k) = (-(m-k)^2)z^{n-m+k}.$
 		Then we adapt the previous computations:
 		\begin{eqnarray*}
 			\sum_{k \in \mathbb{Z}} (k^2)^{-s} (XY^* z^k, z^k)_{L^2} & = &  \sum_{k \in \mathbb{Z}} (-(m-k)^2)(k^2)^{-s} (z^{n-m+k}, z^k)_{L^2} \\
 			& = & \left\{ \begin{array}{cl} 0 & \hbox{ if } n \neq m \\
 				-m^2-2\sum_{k \in \mathbb{N}^*}  k^{-2s + 2} + m^2k^{-2s} & \hbox{ if } n=m \end{array} \right.
 		\end{eqnarray*}
 		.
 		By the way, passing from the complex Fourier Basis to the real Fourier basis, $$\tr^\Delta(XY^*) = -2\zeta(-2) (u,v)_{L^2} + (-1 - 2\zeta(0))(u,v)_{H^1_0} = 0.$$
 		\item Since $a$ and $X$ are even-even class, we thave that $\tr^\Delta(aX)=\tr^\Delta(Xa).$ Setting $a = z^n$ and $X = u \frac{d}{dx},$ with $u = z^m,$ we compute 
 		\begin{eqnarray*}
 			\sum_{k \in \mathbb{Z}} ik (k^2)^{-s} (z^n z^{m} z^k, z^k)_{L^2} & = &  \sum_{k \in \mathbb{Z}} ik.k^{-2s} (z^{n+m+k}, z^k)_{L^2} \\
 			& = & \left\{ \begin{array}{cl} 0 & \hbox{ if } n \neq -m \\
 				\sum_{k \in \mathbb{Z}} ik. k^{-2s} = 0 & \hbox{ if } n=-m \end{array} \right.
 		\end{eqnarray*}
 	By the way, $$\tr^\Delta(aX) = \tr^\Delta(Xa)=0.$$
 	%	\item We take the notations of the proof of (\ref{4.}). Let us first recall that  $p_+$ and $\epsilon(D)$ are self-adjoint and commute with $\Delta^{-s}.$ This enables to calculate the following traces more easily: 
 	%	\begin{eqnarray*}
 	%		\tr\left(\Delta^{-s} p_+ XY\right) & = & \tr\left(p_+ \Delta^{-s}  XY\right)\\
 	%		& = & \sum_{k \in \mathbb{Z}}(imk-k^2) (k^2)^{-s} (z^{n+m+k}, p_+ z^k)_{L^2}\\
 	%		& = & \left\{ \begin{array}{cl} 0 & \hbox{ if } n \neq m \\
 	%			\sum_{k \in \mathbb{N}^*} imk^{-2s+1} - k^{-2s + 2} & \hbox{ if } n=-m \end{array} \right.
 	%		\end{eqnarray*} 
 	%	and 
 	%	\begin{eqnarray*}
 	%		\tr\left(\Delta^{-s} \epsilon(D) XY\right) & = & \tr\left(\epsilon(D) \Delta^{-s}  XY\right)\\
 	%		& = & \sum_{k \in \mathbb{Z}}(mk+k^2) (k^2)^{-s} (z^{n+m+k}, \epsilon(D) z^k)_{L^2}\\
 	%		& = & \left\{ \begin{array}{cl} 0 & \hbox{ if } n \neq m \\
 	%			\sum_{k \in \mathbb{Z}^*} m|k|^{-2s+1} & \hbox{ if } n=-m \\\end{array} \right.\\
 	%		& = & \left\{ \begin{array}{cl} 0 & \hbox{ if } n \neq m \\
 	%			 2m\sum_{k \in \mathbb{N}^*} k^{-2s + 1} & \hbox{ if } n=-m \end{array} \right.
 	%	\end{eqnarray*} 
 	%Using the same methods, we obtain that $$\sum_{k \in \mathbb{N}^*} k^{-2s + 1} = \tr\left(\Delta^{-s}p_+D\right).$$
 	%By the way, $$\tr^\Delta(\epsilon(D)XY) = 
 	%\tr^\Delta\left(p_+D\right)\left(\frac{du}{dx},v\right)_{L^2}$$
 	%and $$\tr^\Delta(p_+XY) = 
 	%\tr^\Delta\left(p_+D\right)\left(\frac{du}{dx},v\right)_{L^2}   -\zeta(2) (u,v)_{L^2}$$
 %which is coherent with (\ref{4.})
\end{enumerate}
 \end{proof}
\subsection{Extension of the Hilbart-Schmidt metric  to $FCl.$}

\begin{Theorem}
	The Hilbert-Schmidt definite positive Hermitian product $$ \left( A,B\right)_{HS} = \tr\left(AB^*\right)$$
	which is positive, definite metric on $Cl^{-1}(S^1,V)$ extends:
	\begin{itemize}
		\item to a Hermitian, non degenerate form on $Cl(S^1,V)$ by $(A,B) \mapsto (A,B)_{\Delta}=\tr^\Delta(AB^*)$
		\item to a Hermitian, non degenerate form on  $Cl_{ee}(S^1,V)$ by $(A,B) \mapsto (A,B)_{\Delta}=\tr^\Delta(AB^*)$
		\item to a $(\R-)$ bilinear, symmetric non degenerate form on $Cl(S^1,V) \oplus Vect(S^1) $ by $(A,B) \mapsto \mathfrak{Re}(A,B)_{\Delta} =\mathfrak{Re}\left(\tr^\Delta(AB^*)\right)$ where $A = a+u,$ $B = b + v,$ with $(a,b) \in Cl(S^1,V)$ and $(u,v) \in Vect(S^1).$ 
	\end{itemize}
	\end{Theorem}
\begin{proof}
	
	We proceed set by set, following the order of the statement of the Theorem.
	
	 \underline{On $Cl(S^1,V)$}:
		The formula $\tr\left(AB^*\right)$ gives obviously a sesquilinear form.
		We prove first that the form is Hermitian: 
		Let $(A,B)\in Cl(S^1,V). $
		$$ \tr^\Delta(BA^*) = \tr^\Delta\left((AB^*)^*\right) = \overline{\tr^\Delta\left(AB^*\right)}.$$
		 Let us then prove that it is non-degenerate. Let $A \in Cl(S^1,V),$ let $u \in C^\infty(S^1,V) \cap \left(Im A - \{0\} \right)$ which is the image of a function $x$ such that $||x||_{L^2} = 1, $ and let $p_x$ be the $L^2-$ orthogonal projection on the $\mathbb{C}-$vector space spanned by $x.$ Then, let $(e_k)_{k \in \N}$ be an orthonormal base with $e_0=x.$ For $\mathfrak{Re}(s) \geq 2ord(A) + 2,$ we observe, applying commutation relations of the usual trace of trace-class operators, the following: 
	\begin{eqnarray*} \tr\left(A \left(p_xA\right)^*\Delta^{-s}\right)& = &\tr\left(\Delta^{-s/2}A \left(Ap_x\right)^*\Delta^{-s/2}\right) \\ & =& \tr\left( \left(Ap_x\right)^*\Delta^{-s/2}.\Delta^{-s/2}A\right) \\ & = &\tr\left( \left(Ap_x\right)^*\Delta^{-s}A\right).\end{eqnarray*}
		By the way, the meromorphic continuation to $\mathbb{C}$ of $s \mapsto \tr\left( \left(Ap_x\right)^*\Delta^{-s}A\right)$ exist and coincide with the meromorphic continuation of $s \mapsto \tr\left(A \left(Ap_x\right)^*\Delta^{-s}\right)$, in particular at $s=0.$
		\begin{eqnarray*}
			\tr\left( \left(p_xA\right)^*\Delta^{-s} A\right) & = & \sum_{k \in \mathbb{N}}  (\left(p_xA\right)^*\Delta^{-s} A e_k, e_k)_{L^2} \\
			& = &  \sum_{k \in \mathbb{N}}  (\Delta^{-s} A e_k, Ap_x e_k)_{L^2}\\
			& = & (\Delta^{-s} A x, Ax)_{L^2}\\
			& = & (\Delta^{-s/2}u,\Delta^{-s/2}u)_{L^2}
			\end{eqnarray*}
		By the way, since $\lim_{s \rightarrow 0} \Delta^{-s/2} = Id$ for weak convergence,
		\begin{eqnarray*}
			\tr^\Delta \left(A \left(Ap_x\right)\right)& = & \lim_{s \rightarrow 0}\tr\left(A \left(Ap_x\right)^*\Delta^{-s}\right)\\
			& = & \lim_{s \rightarrow 0}(\Delta^{-s/2}u,\Delta^{-s/2}u)_{L^2}\\
			& = & ||u||_{L^2}^2\\
			& \neq & 0.			\end{eqnarray*}
		The operator $Ap_x$ is a smoothing (rank 1) operator, and hence is in  $Cl(S^1,V),$ which ends the proof.	
	
	 \underline{On $Cl_{ee}(S^1,V)$ and on $Cl(S^1,V) \oplus Vect(S^1) $}:
		The same arguments as before hold for non-degeneracy, both on $Cl_{ee}(S^1,V)$ and on $Cl(S^1,V) \oplus Vect(S^1) $.  The rest of the arguments follow from the inclusions $Cl_{ee}(S^1,V) \subset Cl(S^1,V)$ and $Cl(S^1,V) + Vect(S^1) = Cl(S^1,V).$ 
	
	\end{proof}
\begin{rem}
	We remark that $(.,.)_\Delta$ is bilinear, non degenerate but not positive. Indeed, from relation (\ref{3.}) of Lemma \ref{calcultr},  $C^\infty(S^1, M_n(\mathbb{C}))$ is an isotropic Lie subalgebra for $(.,.)_\Delta$ which proves that this $\R-$bilinear symmetric form is not positive. %In the other case, $(.,.)_\Delta$ is not positive even if non degenerate on $C^\infty(S^1, M_n(\mathbb{C})).$ 
\end{rem}

From the Lie algebra $Cl(S^1,V) \oplus Vect(S^1), $ we then span by right-invariant action  of $FCl^{*}(S^1,V)$ on $TFCl^{*}(S^1,V)$ a right-invariant pseudo-metric. For this goal, the Lie algebra elements are identified as infinitesimal paths, and actions and Lie brackets are those derived from the coadjoint action (and right-Lie bracket) of $FCl^{*}(S^1,V)$ on $Cl(S^1,V) \oplus Vect(S^1), $ while we consider the trivial mapping defined by the sum $Cl(S^1,V) \oplus Vect(S^1) \rightarrow  Cl(S^1,V)  = Cl(S^1,V) + Vect(S^1) $ in order to compute $\mathfrak{Re}(.;.)_\Delta.$ The same constructions hold for the pseudo-Hermitian metric $(.;.)_\Delta$ on $Cl^*(S^1,V).$

\begin{Definition}
	Let $A \in FCl^{0,*}(S^1,V)$ and let $a \in Cl^0(S^1,V) \oplus Vect(S^1). $ We note by $R_A(a)$ the (right-)action by composition $$R_A(a) = a \circ A.$$
	Then, identifying $T_A FCl^{0,*}(S^1,V)$ with $R_A\left(Cl^0(S^1,V) \oplus Vect(S^1)\right)$ we set a smooth pseudo-Riemannain metric on $T FCl^{0,*}(S^1,V)$ by defining for $$(a,b) \in \left(Cl^0(S^1,V) \oplus Vect(S^1)\right)^2,$$ and hence for   $(R_A(a),R_A(b))  \in \left(T_A FCl^{0,*}(S^1,V)\right)^2,$ 
	$$ (R_A(a),R_A(b))_{\Delta, A} = (a,b)_\Delta.$$
\end{Definition}

%We remark that Equation (\ref{Liebr}) resumes to the restriction to $Cl^0(S^1,V) \oplus Vect(S^1)$ of the classical Lie bracket on $Cl(S^1,V).$ 
%\subsection{Metrics from polarized traces and Sobolev-type metrics}

\section{In search of pseudo-Hermitian connections for $(.,.)_{\Delta}$} \label{s:conn}
There exists some difficulties in describing the whole space of connection 1-forms $\Omega^1(FCl_{Diff(S^1)}(S^1,V),Cl(S^1,V)\rtimes Vect(S^1)).$ Indeed the space of smooth linear maps acting on $Cl(S^1,V)$ is actually not well-understood to our knowledge. In particular, finding an adjoint of the adjoint map for $(.,.)_\Delta$ fails apparently due to the non-traciality of $\tr^\Delta.$ We consider here a class of connections where this smooth linear endomorphism is defined by composition by a smoothing operator. The resulting technical simplifications enables us to get pseudo-Hermitian connections for $(.,.)_{\Delta}.$ Most of them can be easily adapted to get pseudo-Riemannian connections for $\mathfrak{Re}(.,.)_{\Delta}.$  
\subsection{A class of connections}
 Let us define now, for $w \in Cl(S^1,V)$ such that $\forall (a,b)\in Cl(S^1,V),$ 
$$\Theta^w_ab = b[a,w].$$
-

\begin{Proposition}
	The curvature of $\Theta^w$ reads as $$\Omega_{\Theta^w}(a,b)c =[wbw,a] - [waw,b] + [wa,wb]  - [wb,aw] -[[b,a]w] $$ for $(a,b) \in \F Cl(S^1,V).$
\end{Proposition}

\begin{proof}
	Let $(a,b,c)\in Cl(S^1,V).$ 
\begin{eqnarray*}
	\Omega_{\Theta^w}(a,b)c & = & \left[\Theta^w_a,\Theta^w_b\right]c - \Theta^w_{[a,b]}c \\
	& = & c[b,w][a,w] - c[a,w][b,w] -c [[b,a],w] \\
	& = & c[wbw,a] - c[waw,b] + c[wa,wb]  - c[wb,aw] -c[[b,a]w]
	\end{eqnarray*}
	\end{proof}

%\begin{Proposition}
%	The torsion of $\Theta^w$ reads as $$T_{\Theta^w}(a,b) = [b,a](w-1) + bwa -awb$$
%\end{Proposition}
%\begin{proof}
%	\begin{eqnarray*}
%	T_{\Theta^w}(a,b) & = & \Theta^w_ab - \Theta^wba - [b,a]\\
%	& = & b[a,w] - a[b,w] - [b,a] \\
%	& = & [b,a](w-1) + bwa -awb
%	\end{eqnarray*}
%	\end{proof}

Let us analyze 
the connection $\Theta^w$ with $w = i\epsilon(D).$   
\begin{Theorem}
	$\Theta^{i\epsilon(D)}$ is a $Cl^{-\infty}(S^1,V)-$valued connection. 
\end{Theorem}
\begin{proof}
	It follows directly from the fact that $[a,i\epsilon(D)] \in Cl^{-\infty}(S^1,V).$
	\end{proof}
\subsection{Pseudo-Hermitian connections associated with a skew-adjoint pseudodifferential operator}

Let $w \in Cl(S^1,V)$ such that $w^* = -w.$ For example, one can consider the example $w = i\epsilon(D).$
\begin{Lemma}
	$\forall a \in Cl(S^1,V),$
	$\forall w \in Cl(S^1,V)$ such that $w^*=-w, $ $\Theta^w_{a^*}$ is the adjoint of $\Theta^s_a$ for $(.,.)_\Delta$
\end{Lemma}
\begin{proof}
	Let $(a,b,c)\in Cl^{-\infty}(S^1,V)^3$
	\begin{eqnarray*}
		\left(\Theta^w_a b, c)\right)_\Delta & = & \tr^\Delta\left( b[a,w]c^* \right)\\
		& = & \tr^\Delta\left( b([a^*,w])^*c^* \right)\\
		& = & \tr^\Delta\left( b(c[a^*,w])^* \right)\\
		& = & \tr^\Delta\left( b(\Theta_{a^*}c)^* \right)
		\end{eqnarray*}
	\end{proof}
Let us now analyze

$$ (a,b)\in Cl(S^1,\C)^2 \mapsto \theta^{w}_ab=b[a-a^*,w] = (\Theta^w_a-\Theta^w_{a^*})(b).$$
\begin{Theorem}
	$\theta^{w}$ is the connection 1-form of a pseudo-Hermitian connection of $(.,.)_\Delta.$
\end{Theorem} 
\begin{proof}
	Let $ (a,b,c)\in Cl(S^1,\C)^3.$
	
	\begin{eqnarray*}
		\left(\theta^{w}_ab,c\right)_\Delta & = & \left((\Theta^{w}_a - \Theta^w_{a^*})b,c\right)_\Delta\\
		& = & \left((\Theta^{w}_a ,c\right)_\Delta- \left((\Theta^w_{a^*})b,c\right)_\Delta\\
		& = & \left(( b,\Theta^w_{a^*}c\right)_\Delta- \left(b,\Theta^{w}_ac\right)_\Delta\\
		& = & - \left(b,\theta^{w}_{a^*}c\right)_\Delta
	\end{eqnarray*}
	Hence
	$$\left(\theta^{w}_ab,c\right)_\Delta + \left(b,\theta^{w}_{a^*}c\right)_\Delta = 0.$$
\end{proof}

\begin{rem}
	$\theta^{i\epsilon(D)}$ is $Cl^{-\infty}(S^1,V)-$valued as $\Theta^{i\epsilon(D)}$ is.
\end{rem}

\subsection{On another class $Cl^{-\infty}(S^1,V)-$connections}
Motivated by the previous example of $Cl^{-\infty}(S^1,V),$ let us now give families of $Cl^{-\infty}(S^1,V)-$connections which a priori do not include the connections $\theta^{i\epsilon(D)}$ and $\Theta^{i\epsilon(D)}$.  Let us define now, for $s \in Cl^{-\infty}$ and $\forall (a,b)\in Cl(S^1,V),$ 
$$\Theta^{s,l}_ab = sas^*b,$$
$$\Theta^{s,r}_ab = bsas^*,$$
and 
$$\Theta^{s,[]}_ab = \left[sas^*,b\right].$$
Let us describe here their associated class of pseudo-Riemannian connections for $(.,.)_\Delta$ along the lines of the previous section.
Let $s \in Cl^{-\infty}(S^1,V)$ be a smoothing operator. Let $a,b \in  Cl(S^1,V)^2$ and let  $$\theta_a^{s,[]} b= \Theta_a^{s,[]}b - \Theta_{a^*}^{s,[]}b =\left[s(a-a^*)s^*,b\right] ,$$ 
$$\theta_a^{s,l} = \Theta_a^{s,l}b - \Theta_{a^*}^{s,l}b=s(a-a^*)s^*b,$$
and 
$$\theta_a^{s,r} = \Theta_a^{s,r}b - \Theta_{a^*}^{s,r}b= bs(a-a^*)s^*.$$
\begin{Lemma} \label{Thetas}
	$\forall a \in Cl(S^1,V),$
	$\forall s \in Cl^{-\infty}(S^1,V),$ 
	\begin{itemize}
		\item $ \Theta^{s,l}_{a^*}$ is the adjoint of $\Theta^{s,l}_a$ for $(.,.)_\Delta$
		\item $ \Theta^{s,r}_{a^*}$ is the adjoint of $\Theta^{s,l}_a$ for $(.,.)_\Delta$
		\item $ \Theta^{s,[]}_{a^*}$ is the adjoint of $\Theta^{s,l}_a$ for $(.,.)_\Delta$
	\end{itemize} 
\end{Lemma} 
\begin{proof}
		Let $ (a,b,c)\in Cl(S^1,\C)^3.$
	\begin{eqnarray*}
		(\Theta_a^{s,l}b,c)_\Delta  & = & (sas^*b ,c)_\Delta \\
		& = & \tr^\Delta( sas^*b c^*)   \\
		& = & \tr^\Delta( b c^*sas^*)  \hbox{ because }sas^*\in Cl^{-\infty}(S^1,V) \\
		& = & \tr^\Delta( b (sa^*s^*c)^*) .
	\end{eqnarray*}
which proved the first point.
\begin{eqnarray*}
	(\Theta_a^{s,r}b,c)_\Delta  & = & (bsas^* ,c)_\Delta \\
	& = & \tr^\Delta( bsas^* c^*)   \\
	& = & \tr^\Delta( b (csa^*s^*)^*) .
\end{eqnarray*}
which proves the second point.
The third pont is proved straightway by the remark $\Theta_a^{s,[]}=\Theta_a^{s,l}-\Theta_a^{s,r}$
\end{proof}
 
\begin{Theorem}
	
	  Then  $\theta^{s,[]},$ $\theta^{s,l}$ and $\theta^{s,r},$ define three right-invariant pseudo-Hermitian connections on $FCl(S^1,V).$ 
\end{Theorem}
\begin{proof} It follows from Lemma \ref{Thetas}
	\end{proof}
%Let us now compute the torsion of this pseudo-Riemannian connection
%\begin{Proposition}
%	The torsion of $\theta^{s,[]}$ is:
%$$T^s(a,b) = \theta^s_a(b) - \theta^s_b(a) - [b,a] = \left[s(a-a^*)s^* + \frac{1}{2}a,b\right] - \left[s(b-b^*)s^* + \frac{1}{2}b,a\right].$$
%\end{Proposition}
%\begin{proof}
%Let $(a,b)\in  Cl(S^1,\C)^2.$
%\begin{eqnarray*}
%\theta^s_a(b) - \theta^s_b(a) - [b,a] & = & \left[sas^*,b\right] - \left[sa^*s^*,b\right]	- \left[sbs^*,a\right] + \left[sb^*s^*,a\right] -[b,a] \\
%& = & \left[s(a-a^*)s^* + \frac{1}{2}a,b\right] - \left[s(b-b^*)s^* + \frac{1}{2}b,a\right].
%\end{eqnarray*}	
%	\end{proof}

%\begin{Proposition}
%	Let $(a,b)\in  Cl(S^1,\C)^2.$ The curvature of $\theta^{s,[]}$ is
%	$$\Omega^{s,[]}(a,b) = -\left[s([a,b] - [b^*,a^*])s^*+\left[s(b - b^*)s^*,s(a - a^*)s^*\right],.\right] $$
%\end{Proposition}

%\begin{proof}
%	Let $(a,b,c)\in  Cl(S^1,\C)^3.$ We compute 
%	\begin{eqnarray*}
%		\Omega^s(a,b)c  & = & \left[\theta^s_a,\theta^s_b\right]c -\theta^s_{[b,a]}c\\
%	& = & \left[s(a - a^*)s^*,\left[s(b-b^*)s^*,c\right]\right] - \left[s(b - b^*)s^*,\left[s(a-a^*)s^*,c\right]\right] \\&&	+ \left[s([a,b] - [b^*,a^*])s^*,c\right]	
%	\\
%& = & \left[s(a - a^*)s^*,\left[s(b-b^*)s^*,c\right]\right] 
%+ \left[c,\left[s(b - b^*)s^*,s(a-a^*)s^*\right]\right]
%\\
%&& + \left[s(a - a^*)s^*,\left[c, s(b-b^*)s^*\right]\right]	+ \left[s([a,b] - [b^*,a^*])s^*,c\right] \\
%& = & \left[c,\left[s(b - b^*)s^*,s(a-a^*)s^*\right]\right] 	- \left[c,s([a,b] - [b^*,a^*])s^*\right] \end{eqnarray*}
%\end{proof}
%\section{

\section{Last remarks}
\subsection{$Diff(S^1)$ versus $Diff_+(S^1).$}
The group of diffeomorphims $Diff(S^1)$ splits into two connected components 
$$Diff(S^1)=Diff_+(S^1)\coprod Diff_-(S^1)$$
where $ Diff_-(S^1)$ is the space of diffeomorphisms which reverse the orientation of $S^1.$ Among these diffeomorphisms, there is the conjugate map $$  z \in S^1 \mapsto \bar z $$
which induces an involution 
$$ Conj: (z \mapsto f(z)) \in L^2(S^1,V) \mapsto (z \mapsto f(\bar z)) \in L^2(S^1,V)$$
that decomposes blockwise in $L^2(S^1,V) = E_- \oplus E_0 \oplus E_+$ as
$$ Conj = \left(\begin{array}{ccc}
0 & 0 & * \\
0 & Id_{E_0} & 0 \\
* & 0 & 0
\end{array}\right)$$
and in $H_+ \oplus H_-,$ for any choice made for $H_+$ and $H_-$ (see section \ref{sect2}) as

$$ Conj = \left(\begin{array}{cc}
Conj_{++} & Conj_{+-}  \\
Conj_{-+} & Conj_{--}  \\
\end{array}\right)$$
where $Conj_{++}$ and $ Conj_{--}$ are finite rank, smoothing operators. We have that $$Diff_-(S^1) = Conj \circ Diff_+(S^1) = Diff_+(S^1) \circ Conj.$$ By the way, 
given $A \in FCl_{Diff(S^1)}(S^1,V),$ if the phase diffeomorphism $g$ of $A$ is orientation preserving, then, under the blockwise decomposition $H_+ \oplus H_-,$
$$ A = \left(\begin{array}{cc}
A_{++} & A_{+-}  \\
A_{-+} & A_{--}  \\
\end{array}\right)$$
$ A_{+-}$ and $A_{-+}$ are smoothing operators according to \cite{Ma2016}, and if $g \in Diff_-(S^1),$  $ A_{++}$ and $A_{--}$ are smoothing operators.
\subsection{The Schwinger cocycle and the connection $\Theta^{i\epsilon}.$}
Let us make the two following remarks
\begin{Proposition}
	Let $(a,b) \in Cl(S^1,V)\rtimes Vect(S^1).$ Then 
	$$c_s(a,b) = -i\tr^\Delta(\Theta^{i\epsilon}_ab) = \tr^\Delta(\Theta^{i\epsilon}_a\Theta^{i\epsilon}_b \epsilon(D)).$$ 
\end{Proposition}
\begin{proof}
	We compute independently
	$$\tr^\Delta(\Theta^{i\epsilon}_ab) = \tr^\Delta(b[a,i\epsilon(D)]) = i \tr^\Delta([a,\epsilon(D)]b)$$
	since $[a,\epsilon(D)] \in Cl^{-\infty}(S^1,V),$ and
	\begin{eqnarray*}
		\tr^\Delta(\Theta^{i\epsilon}_a\Theta^{i\epsilon}_b \epsilon(D)) & = & \tr^\Delta( \epsilon(D)[b,i\epsilon(D)][a,i\epsilon(D)]) \\
		& = & - c_s(b,a)\\
		& = & c_s(a,b).
		\end{eqnarray*}
	\end{proof}
\begin{rem}
	When defining a smoothing connection $\theta$ on $Cl(S^1,V)\rtimes Vect(S^1),$ we define a map with values on the first component of the product $Cl(S^1,V)\times Vect(S^1).$
\end{rem}  

*

:
:\begin{Theorem}
	The Schwinger cocycle $c_s$ has the same cohomology class as $$c_1^{i\epsilon} : (a,b) \in Cl(S^1,V)^2 \mapsto  \frac{1}{2}\tr^\Delta\left(\Omega^{i\epsilon}(a,b)\epsilon(D)\right)$$
	where $\Omega^{i\epsilon}$ is the curvature of $\Theta^{i\epsilon}.$
\end{Theorem}
\begin{proof}
	We have, $\forall (a,b) \in Cl(S^1,V)^2: $
	\begin{eqnarray*}
		\tr^\Delta\left(\Omega^{i\epsilon}(a,b)\epsilon(D)\right) & = & -\tr^\Delta\left(\epsilon(D)[b,\epsilon(D)][a,\epsilon(D)] - \epsilon(D)[a,\epsilon(D)][b,\epsilon(D)]\right. \\ && \left.- \epsilon(D)[[b,a],\epsilon(D)] \right) \\
		& = & -\tr^\Delta\left(\epsilon(D)[b,\epsilon(D)][a,\epsilon(D)] \right) + \tr^\Delta\left( \epsilon(D)[a,\epsilon(D)][b,\epsilon(D)] \right) \\ && - \tr^\Delta\left( \epsilon(D)[[b,a],\epsilon(D)] \right) \\
		& = & c_s(a,b) + c_s(a,b) + (\delta \gamma)(a,b)
		\end{eqnarray*}
	where $\delta$ is the coboundary operator and $$\gamma: a \in Cl(S^1,V) \mapsto \tr^\Delta\left( \epsilon(D)[a,\epsilon(D)] \right) = c_s(a,\epsilon(D)).$$
	\end{proof}
%Let us now generalize the family of cocycles initiated by $ c_1^{i\epsilon}.$
%\begin{Theorem}
%	Let $k \in \N^*.$
%	Let us define  $$c_k^{i\epsilon} : (a_1,...,a_{2k}) \in Cl(S^1,V)^{2k} \mapsto  \frac{1}{2^k}\tr^\Delta\left(\left(\Omega^{i\epsilon}\right)^k(a_1,...,a_{2k})\epsilon(D)\right).$$
%	Then $c_k^{i\epsilon}$ is a closed cocycle on $Cl(S^1,V)$ with non-vanishing cohomology class. 
%\end{Theorem}
%\begin{proof}
	
%	\end{proof}
\subsection{On even-even $Diff(S^1)-$pseudo-differential operators}
Considering now $$FCl_{ee,Diff(S^1)}(S^1,V) = Cl^*_{ee}(S^1,V) \rtimes  Diff(S^1),$$
we remark that the renormalized trace $\tr^\Delta$ is tracial on its Lie algebra $Cl_{ee}(S^1,V) \rtimes  Vect(S^1),$ i.e. $$\forall (a,b)\in Cl_{ee}(S^1,V), \quad \tr^\Delta([a,b])=0$$
(representing $Cl_{ee}(S^1,V) \rtimes  Vect(S^1)$ in $Cl_{ee}(S^1,V)$ as in the rest of the text).This enables to state the following property:
\begin{Proposition}
	 $\forall a \in Cl_{ee}(S^1,V),$ the adjoint map 
	$$ ad_a : b \mapsto ad_a b = [b,a]$$
	has an adjoint map for $(.,.)_\Delta$ given by 
	$$ ad_a^* = ad_{a^*}.$$
\end{Proposition}  
\begin{proof}
	Let $(a,b,c) \in Cl_{ee}(S^1,V)^3.$
	\begin{eqnarray*}
	\tr^\Delta\left((ad_a b)c^*\right) & = & \tr^\Delta\left([b,a]c^*\right) \\
	& = & \tr^\Delta\left(bac^*\right) - \tr^\Delta\left(abc^*\right) \\
	& = & \tr^\Delta\left(bac^*\right) - \tr^\Delta\left(bc^*a\right) \\
		& = & \tr^\Delta\left(b(ca^*)^*\right) - \tr^\Delta\left(b(a^*c)^*\right) \\
		& = & \tr^\Delta\left(b([c,a^*])^*\right)
	\end{eqnarray*}
	\end{proof}
As a consequence, applying the arguments of \cite{Freed1988} and especially those leading to \cite[Proposition 1.7]{Freed1988} to \textbf{right-}invariant vector fields on $FCl_{ee,Diff(S^1)}(S^1,V)$, we get:
\begin{Theorem}
	The pseudo-Riemannian metric $\mathfrak{Re}(.,.)_\Delta$ admits a unique pseudo-Riemannian, torsion-free (i.e. Levi-Civita) connection $\nabla^{\Delta}$ that reads as 
	$$\nabla^\Delta_ab = \frac{1}{2}\left( ad_a b  - ad_{a^*} b - ad_{b^*}a\right)$$
\end{Theorem}

\subsection{On bounded even-even $Diff(S^1)-$pseudo-differential operators} 

Let us finish our remarks with the group of ($L^2-$)bounded even-even $Diff(S^1)-$pseudo-differential operators. Its Lie algebra 
$$ Cl^0_{ee}(S^1,V) \rtimes Vect(S^1)$$ also reads as
$$ Cl^{-1}_{ee}(S^1,V) \oplus DO^0(S^1,V) \oplus (Vect(S^1)\otimes Id_V).$$
and the pseudo-Riemannian product $\mathfrak{Re}(.,.)_\Delta$ decomposes blockwise as
$$\left(\begin{array}{ccc}
\mathfrak{Re}(.,.)_{HS} & * & * \\
* & 0 & 0 \\
* & 0 & 0 
\end{array}\right),$$ where $\mathfrak{Re}(.,.)_{HS} = \mathfrak{Re}\left((.,.)_{HS}\right)$ is the scalar product derived from the Hilbert-Schmidt Hermitian product $(.,.)_{HS}.$


\begin{thebibliography}{99}
 







\bibitem{ARS1} Adams, M.; Ratiu, T.; Schmidt, R.; A Lie group structure for pseudodifferential operators;
\textit{Math. Annalen} \textbf{273},   529-551 (1986).

\bibitem{ARS2} Adams, M.; Ratiu, T.; Schmidt, R.; A Lie group structure for Fourier integral operators;
	\textit{Math. Annalen} \textbf{276}, no.1 , 19--41 (1986).

\bibitem{Adl} Adler, M.; On a trace functionnal for formal pseudo-differential
 operators and the symplectic structure of Korteweg-de Vries type equations
\textit{Inventiones Math.} \textbf{50} 219-248 (1979)


\bibitem{BN2005} Batubenge, A.; Ntumba, P.; On the way to Fr\"olicher Lie groups
\textit{Quaestionnes mathematicae} (2005) \textbf{28} no1, 73--93
\bibitem{B} Berger, M.; \textit{A panoramic overview of Riemannian geometry} Springer (2003)
\bibitem{BK}  Bokobza-Haggiag, J.; { Op\'erateurs pseudo-diff\'erentiels sur une vari\'et\'e diff\'erentiable}; {\it Ann. Inst. Fourier, Grenoble} \textbf{19,1}  125-177 (1969)



\bibitem{CDMP} Cardona, A.; Ducourtioux, C.; Magnot, J-P.; Paycha, S.;
Weighted traces on pseudo-differential operators and geometry on loop groups;
\textit{Infin. Dimens.
Anal. Quantum Probab. Relat. Top.} \textbf{5} no4 503-541 (2002)

\bibitem{CDP} Cardona, A.; Ducourtioux, C.; Paycha, S.; From tracial anomalies to anomalies in quantum field theory \textit{Comm. Math. Phys.} \textbf{242} no 1-2 31-65 (2003)

\bibitem{cen} Cederwall, M.; Ferretti, G.; Nilsson, B; Westerberg, A.; Schwinger terms and 
cohomology of pseudo-differential operators \textit{Comm. Math. Phys.} \textbf{175}, 203-220 (1996)

%\bibitem{CSW}  Christensen, D.; Sinnamon, G.; Wu, E.; 
%The
%D-topology  for  diffeological  spaces
%,  \textit{Pacific  J.  Math.} (2014)
%\textbf{272}
%no1,
%87–-110.

%\bibitem{CW} Christensen, D.; Wu, E.; Tangent spaces and tangent bundles for
%diffeological spaces, \emph{Cahiers de Topologie et G\'eom\'etrie Diff\'erentielle} (2016) Volume LVII 3-50.

%\bibitem{DN2007-1} Dugmore, D.; Ntumba, P.;On tangent cones of Fr\"olicher spaces
%\textit{Quaetiones mathematicae} (2007) \textbf{30} no.1,
%67--83.


\bibitem{Ee} Eells, J.;  A setting for global analysis
\textit{Bull. Amer. Math. Soc.} {\bf 72} 751-807 (1966)

\bibitem{Freed1988} Freed, D.; The Geometry of loop groups \textit{J. Diff. Geome.} \textbf{28} 223-276 (1988)

\bibitem{FK} Fr\"olicher, A; Kriegl, A; {\it Linear spaces and differentiation theory} (1988) Wiley series in Pure and Applied Mathematics, Wiley Interscience 

\bibitem{Gil} Gilkey, P;
{\it Invariance theory, the heat equation and the Atiyah-Singer index theorem}
Publish or Perish (1984)

\bibitem{Horm} H\"ormander,L.; Fourier integral operators. I; \textit{Acta Mathematica} \textbf{127} 79-189 (1971)

\bibitem{Hir} Hirsch, M.; \textit{Differential Topology} (1997),  Springer 

\bibitem{Igdiff} Iglesias-Zemmour, P.
\textit{Diffeology} 
Mathematical Surveys and Monographs \textbf{185} AMS  (2013).

\bibitem{Ka} Kassel, Ch.;
Le r\'esidu non commutatif (d'apr\`es M. Wodzicki)  S\'eminaire
Bourbaki, Vol. 1988/89. \textit{Ast\'erisque} {\bf 177-178},
Exp. No. 708, 199-229 (1989)

\bibitem{KW} Khesin, B.; Wendt, R.; \textit{The Geometry of Infinite-Dimensional Groups} Springer Verlag (2009)

\bibitem{KV1} Kontsevich, M.; Vishik, S.;
{ Determinants of elliptic pseudo-differential operators} Max
Plank Institut fur Mathematik, Bonn, Germany, preprint n. 94-30
(1994)

\bibitem{KV2}  Kontsevich, M.; Vishik, S.; Geometry of determinants of elliptic operators.
Functional analysis on the eve of the 21st century, Vol. 1 (New Brunswick, NJ, 1993), 
\textit{Progr. Math.} \textbf{131},173-197
(1995)

\bibitem{KK} Kravchenko, O.S.; Khesin, B.A.; A central extension of the algebra of pseudo-differential 
symbols \textit{Funct. Anal. Appl.} \textbf{25} 152-154 (1991) 

\bibitem{KM} Kriegl, A.; Michor,  P.W.; 
\textit{The convenient setting for global analysis} (1997); 
AMS Math. Surveys and Monographs \textbf{53}, AMS, Providence 

\bibitem{KMR} Kriegl, A.; Michor, P. W.; Rainer, A.; An exotic zoo of diffeomorphism groups on $\mathbb{R}^n$. \textit{Ann. Global Anal. Geom.} (2015) \textbf{47}  no. 2, 179–-222 . 

\bibitem{KMS} Kolar, I.; Michor, P.W.; Slovak, J.; \textit{Natural operations in differential geometry} (1993); Springer


\bibitem{Lau2011} Laubinger, M.; A Lie algebra for Fr\"olicher groups \textit{Indag. Math.} \textbf{21} no 3-4, 156--174 (2011) 

\bibitem{Le} Lesch, M.; On the non commutative residue for pseudo-differential operators 
with log-polyhomogeneous symbol \textit{Ann. Glob. Anal. Geom.} \textbf{17} 151-187 (1998)

	\bibitem{Les} Leslie, J.; On a Diffeological Group Realization of certain Generalized symmetrizable Kac-Moody 
Lie Algebras \textit{J. Lie Theory} \textbf{13} (2003),
427-442.

\bibitem{Ma2003}  Magnot, J-P.; The K\"ahler form on the loop group and the Radul cocycle on Pseudo-differential Operators; \textit{GROUP'24: Physical and Mathematical aspects of symmetries}, Proceedings of the 24th International Colloquium on Group Theorical Methods in Physics, Paris, France, 15-20 July 2002; Institut of Physic conferences Publishing \textbf{173}, 671-675, IOP Bristol and Philadelphia (2003) 

\bibitem{Ma2006} Magnot, J-P.; Chern forms on mapping spaces,
\textit{Acta Appl. Math.} \textbf{91}, no. 1, 67-95 (2006).

\bibitem{Ma2006-2}  Magnot, J-P.; Renormalized traces and cocycles on the algebra 
of $S^1$-pseudo-differential operators; \textit{Lett. Math. Phys.} \textbf{75} no2, 111-127 (2006)



\bibitem{Ma2006-3} Magnot, J-P.; Diff\'eologie du fibr\'e d'Holonomie en dimension infinie,
\textit{ C. R. Math. Soc. Roy. Can.} \textbf{28} no4 (2006)
121-127.
\bibitem{Ma2008} Magnot, J-P.; The Schwinger cocycle on algebras with unbounded operators. 
%\textit{Bull. Sci. Math.} \textbf{132}, no. 2, 112-127 (2008).
\bibitem{Ma2013} Magnot, J-P.; Ambrose-Singer theorem on diffeological bundles and complete integrability
of KP equations. {\em Int. J. Geom. Meth. Mod. Phys.} {\bf 10}, no 9 (2013) Article ID
1350043.

\bibitem{Ma2015} Magnot, J-P.; {\it q-deformed Lax equations and their differential geometric background}
(2015),
Lambert  Academic Publishing, Saarbrucken, Germany.


\bibitem{Ma2016} Magnot, J-P.; On $Diff(M)-$pseudodifferential operators and the geometry of non linear grassmannians. 
\textit{Mathematics} \textbf{4}, 1; doi:10.3390/math4010001 (2016)

\bibitem{Ma2018-2} Magnot, J-P.; The group of diffeomorphisms of a non-compact manifold is not regular
\textit{Demonstr. Math.} 51, No. 1, 8-16 (2018)

\bibitem{MR2016} Magnot, J-P.; Reyes, E. G.; Well-posedness of the Kadomtsev-Petviashvili hierarchy, 
Mulase factorization, and Fr\"olicher Lie groups  \texttt{arXiv:1608.03994}

\bibitem{MR2018} Magnot, J-P; Reyes E. G.; $Diff_+(S^1)-$pseudo-differential operators and the Kadomtsev-Petviashvili hierarchy \texttt{ArXiv:1808.03791}

\bibitem{MR2019}
Magnot, J-P. Reyes, E.G.; The Cauchy problem of the Kadomtsev-Petviashvili hierarchy and
infinite-dimensional groups.  \textit{in Nonlinear Systems and Their Remarkable Mathematical Structures, Volume 2; Norbert Euler and  Maria Clara Nucci Editors}, CRC press (2019) section B6

\bibitem{Mick} Mickelsson, J.; \textit{Current algebras and groups}. Plenum monographs in Nonlinear Physics, 	Springer (1989)

\bibitem{Mick} Mickelsson, J.; Wodzicki residue and anomalies on current algebras \textit{ Integrable models and strings} A. Alekseev and al. eds. \textit{Lecture notes in Physics} \textbf{436}, Springer (1994)

\bibitem{Neeb2007} Neeb, K-H.; Towards a Lie theory of locally convex groups \textit{Japanese J. Math.} (2006) \textbf{1}, 291-468

\bibitem{Om1973} Omori, H.;
Groups of diffeomorphisms and their subgroups.
\textit{Trans. Amer. Math. Soc.} (1973) \textbf{179} , 85–-122 . 

\bibitem{Om1981}Omori, H.; A remark on nonenlargeable Lie algebras. \textit{J. Math. Soc. Japan} (1981) \textit{33} no. 4, 707-–710 .

\bibitem{Om} Omori, H.; {\it Infinite Dimensional Lie Groups} (1997) AMS Translations of Mathematical Monographs no {\bf 158}  Amer. Math. Soc., Providence, R.I. 

\bibitem{Pay} Paycha, S.; Renormalized traces a looking glass into infinite dimensional geometry
\textit{Infin. Dimens.
Anal. Quantum Probab. Relat. Top.} \textbf{4} 221-226 (2001)

\bibitem{PayBook} Paycha, S; 
\textit{Regularised integrals, sums and traces. An analytic point of view.}
University Lecture Series \textbf{59}, AMS (2012).

\bibitem{Pay2008} Paycha, S.; Paths towards an extension of Chern-Weil
	calculus to a class of infinite dimensional vector
	bundles. in {\it Geometric and Topological Methods for Quantum Field Theory} Cambridge University Press, 81-139 (2013)
\bibitem{PS} Pressley, A.; Segal, G.; {\it Loop Groups} Oxford Univ. Press (1988)

\bibitem{Rad} O.A.Radul; {Lie albegras of differential operators, their central extensions, and W-algebras} \textit{Funct. Anal. Appl.} \textbf{25}, 25-39 (1991)

\bibitem{See} Seeley, R.T.; {Complex powers of an elliptic operator}
\textit{AMS Proc. Symp. Pure Math.} \textbf{10},  288-307 (1968)

\bibitem{Sch}  Schwinger, J.; Field theory of commutators; 
\textit{Phys. Rev. Lett.} \textbf{3}, 296-297 (1959)

	\bibitem{Scott} Scott, S.;
\textit{Traces and determinants of pseudodifferential operators}; 
OUP (2010)

\bibitem{Sou} Souriau, J-M.; un algorithme g\'en\'erateur de structures quantiques \textbf{Ast\'erisque} (hors s\'erie) 341-399 (1985)

\bibitem{Wa} Watts, J.; \textit{Diffeologies, differentiable spaces and symplectic geometry}. University of Toronto,
PhD thesis (2013). arXiv:1208.3634v1.

\bibitem{Wid} Widom, H.; { A complete symbolic calculus for pseudo-differential operators}; {\it Bull. Sc. Math. 2e serie} \textbf{104} 19-63  (1980)

\bibitem{W} Wodzicki, M.; {Local invariants in spectral asymmetry}
\textit{Inv. Math.} {\bf 75}, 143-178  (1984)
\end{thebibliography}
\end{document}